\documentclass[11pt]{amsart}
\usepackage{setspace}
\usepackage{graphicx}
\usepackage{geometry}
\usepackage{multirow}
\usepackage{commath}
\usepackage{mdframed}
\usepackage{amsmath, amsthm}
\usepackage{bbm}
\usepackage{amssymb}
\usepackage{array}
\newcolumntype{M}[1]{>{\centering\arraybackslash}m{#1}}
\usepackage{mathtools}

\usepackage{amsthm}
\usepackage{float}
\usepackage{caption}
\usepackage{color}
\usepackage{enumerate}
\usepackage{tikz}
\usepackage{hyperref}
\hypersetup{
   colorlinks = true, 
   linkcolor = blue, 
   filecolor = magenta,      
   urlcolor = black, 
}
\usepackage{url}
\expandafter\def\expandafter\UrlBreaks\expandafter{\UrlBreaks
  \do\a\do\b\do\c\do\d\do\e\do\f\do\g\do\h\do\i\do\j
  \do\k\do\l\do\m\do\n\do\o\do\p\do\q\do\r\do\s\do\t
  \do\u\do\v\do\w\do\x\do\y\do\z\do\A\do\B\do\C\do\D
  \do\E\do\F\do\G\do\H\do\I\do\J\do\K\do\L\do\M\do\N
  \do\O\do\P\do\Q\do\R\do\S\do\T\do\U\do\V\do\W\do\X
  \do\Y\do\Z}
\usepackage{setspace}
\usepackage{systeme}
\usepackage{tabularx}
\usepackage{cite}
\usepackage{etoolbox}
\usepackage[british]{babel}
\newtheorem{theorem}{Theorem}[section]
\newtheorem{lemma}[theorem]{Lemma}
\newtheorem{cor}[theorem]{Corollary}

\newtheorem{problem}[theorem]{Problem}
\newtheorem{fact}[theorem]{Fact}

\newtheorem{prop}[theorem]{Proposition}

\newtheorem*{claim*}{Claim}
\newtheorem*{theorem*}{Theorem}
\newtheorem*{prop*}{Proposition}
\newtheorem*{lemma*}{Lemma}
\newtheorem*{keyobservation*}{Key Observation}
\newtheorem*{conjecture*}{Conjecture}
\newtheorem*{mainthm*}{Main Theorem}
\numberwithin{equation}{section}
\theoremstyle{definition}

\newtheorem*{notation*}{Notation}
\newtheorem{defn}[theorem]{Definition}
\newtheorem{example}[theorem]{Example}
\newtheorem{remark}[theorem]{Remark}

\newcommand{\R}[0]{\mathbb{R}}
\newcommand{\Q}[0]{\mathbb{Q}}
\newcommand{\N}[0]{\mathbb{N}}
\newcommand{\Z}[0]{\mathbb{Z}}
\newcommand{\M}[0]{\mathcal{M}}

\newcommand{\ELL}[0]{L}
\newcommand{\monster}[0]{U}
\newcommand{\poly}[0]{\mathrm{poly}}
\newcommand{\tp}[0]{\mathrm{tp}}
\newcommand{\fs}[0]{\mathrm{fs}}
\newcommand{\VC}[0]{\mathrm{VC}}
\newcommand{\Av}[0]{\mathrm{Av}}
\newcommand{\timesdots}[0]{\times\cdots\times}
\counterwithout{equation}{section}
\title{Homogeneous hypergraph regularity lemmas via $k$-strong honest definitions}
\author{Mervyn Tong}
\address{Department of Pure Mathematics and Mathematical Statistics, Centre for Mathematical Sciences, Wilberforce Road, Cambridge CB3 0WB, United Kingdom}
\email{hwmt3@cam.ac.uk}
\date{\today}
\begin{document}
\begin{abstract}
    We prove that $(k+1)$-uniform hypergraphs definable in an NIP strongly $k$-distal structure satisfy a homogeneous regularity lemma --- they can be partitioned into a bounded number of simplicial complexes, most of which are homogeneous (meaning that the restriction of the hypergraph to the simplicial complex is either complete or empty). Furthermore, the parts of the partition can be chosen uniformly definably, and the size of the partition is polynomial in the reciprocal of the error parameter. This extends the homogeneous regularity lemma proven by Chernikov and Starchenko for hypergraphs definable in a distal structure.

    We prove this by introducing $k$-strong honest definitions and showing that an NIP structure is strongly $k$-distal if and only if every formula $\varphi(x_1, ..., x_k; y)$ has a $k$-strong honest definition. This extends the theory of strong honest definitions in distal structures to the higher-arity setting.    
\end{abstract}
\maketitle
\section{Introduction}\label{secintro}
Let $G=(V\sqcup V, E)$ be a finite bipartite graph. Szemer\'edi's celebrated regularity lemma says that, given an error parameter $\delta>0$, there is a partition of $V$ whose size is bounded in terms of $\delta$, inducing a partition of $V^2$ into boxes, such that $E$ is `quasirandom' with respect to most of the boxes: that is, the distribution of $E$ within the box is approximately random (equivalently, regular or uniform).
\patchcmd{\thmhead}{(#3)}{#3}{}{}
\begin{theorem}[\cite{szemerediregularitylemma}]
    For all $\delta>0$, there is $K\in\N$ such that the following holds.

    Let $G=(V\sqcup V,E)$ be a finite bipartite graph. Then there is a partition $V=V_1\sqcup \cdots \sqcup V_K$ such that
    \[\sum_{V_i\times V_j\text{ }\delta\text{-regular}}|V_i\times V_j|\geq (1-\delta) |V|^2.\]
\end{theorem}
\patchcmd{\thmhead}{#3}{(#3)}{}{}
Here, `$\delta$-regularity' is the relevant notion of quasirandomness --- see Definition \ref{defndeltaregular}. The size $K$ of the partition is bounded in terms of the error $\delta$ but in general very large: a tower of exponentials whose height is polynomial in $\delta^{-1}$ \cite{gowers1997}.

Generalising Szemer\'edi's regularity lemma to $k$-partite $k$-uniform hypergraphs $(\bigsqcup_{i=1}^k V, E)$ was a surprisingly complicated task. One of the key observations is that $V^k$ should be partitioned not into boxes induced by a partition of $V$ but into `simplicial complexes' induced by partitions of $V^1, ..., V^{k-1}$ (see Section \ref{sechypreg}), of which boxes are a special case. Indeed, it is too much to ask for a general hypergraph to be partitioned into quasirandom boxes; in any such attempt, the associated notion of quasirandomness is too weak to be applicable in certain situations. On the other hand, Nagle, R\"odl, Schacht, and Skokan \cite{nrs, rodlskokan} and Gowers \cite{gowersregularity} (among others --- see \cite{regularityhistory} for an account) have proven that every hypergraph can be partitioned into simplicial complexes, most of which are quasirandom in a strong sense. (The precise notions of quasirandomness are not important for our exposition, so we refer the reader to the papers cited above and \cite{gowershypergraphs}.)

Let us return to graphs for now. Since Szemer\'edi proved his vanilla (read `original', not `pedestrian') regularity lemma, it has been endowed with many flavours: there has been an active strand of research aimed at strengthening the conclusion of Szemer\'edi's regularity lemma by restricting the class of graphs under consideration, such as in \cite{malliarisshelah, semialgebraicregularity, lovaszszegedy, alonfischernewman, regularitylemma}. We are particularly interested in when the notion of quasirandomness can be strengthened to \textit{homogeneity}: given a bipartite graph $(V\sqcup V, E)$, we say that $W\subseteq V^2$ is \textit{($E$-)homogeneous} if the restriction of $E$ to $W$ is complete or empty. We seek classes $\mathcal{G}$ of bipartite graphs that admit a \textit{homogeneous regularity lemma}: for all $(V\sqcup V, E)\in\mathcal{G}$, $V^2$ be partitioned into a bounded number of boxes (induced by a partition of $V$), most of which are homogeneous. Model theory provides an answer through \textit{distality}, as shown by Chernikov and Starchenko. Distality is a refinement of NIP (that is, the setting of bounded VC-dimension), and is a vast generalisation of o-minimality; distal theories are known to have many desirable combinatorial properties (see, for example, \cite{regularitylemma, cuttinglemma, mypaperblms, anderson}). See Section \ref{sechigherarity} for a definition of distality.

\patchcmd{\thmhead}{(#3)}{#3}{}{}
\begin{theorem}[\cite{regularitylemma}]\label{abridgeddistalgraphreg}
    Let $T$ be a distal $L$-theory, $M\models T$, and let $\varphi(x,y)\in L(M)$ with $|x|=|y|=: d$. Then, for each $\delta>0$, there is a natural number $K\leq \poly_\varphi(\delta^{-1})$ such that the following holds.

    Let $V\subseteq M^d$ be finite. Then there is a partition $V=V_1\sqcup\cdots \sqcup V_K$ such that
    \[\sum_{V_i\times V_j\text{ }\varphi\text{-homogeneous}}|V_i\times V_j|\geq (1-\delta) |V|^2.\]
\end{theorem}
\patchcmd{\thmhead}{#3}{(#3)}{}{}

In other words, if $\mathcal{G}$ is the class of finite graphs definable by a fixed binary relation in a distal theory, then $\mathcal{G}$ admits a homogeneous regularity lemma (with the additional feature that the size of the partition is much smaller than in general: it is polynomial in the reciprocal of the error $\delta$). Thus, distal theories constitute a general setting for homogeneous regularity lemmas.

What about hypergraphs? We now seek classes $\mathcal{H}$ of $k$-partite $k$-uniform hypergraphs that admit a \textit{homogeneous regularity lemma}: for all $(\bigsqcup_{i=1}^k V, E)\in\mathcal{H}$, $V^k$ be partitioned into a bounded number of \textit{simplicial complexes} $W$ (induced by partitions of $V^1, ..., V^{k-1}$), most of which are homogeneous (that is, the restriction of $E$ to $W$ is complete or empty). It turns out that the immediate analogue of Theorem \ref{abridgeddistalgraphreg} for hypergraphs holds.

\patchcmd{\thmhead}{(#3)}{#3}{}{}
\begin{theorem}[\cite{regularitylemma}]\label{abridgeddistalhypergraphreg}
    Let $T$ be a distal $L$-theory, $M\models T$, and let $\varphi(x_1, ..., x_k)\in L(M)$ with $|x_1|=\cdots =|x_k|=: d$. Then, for each $\delta>0$, there is a natural number $K\leq \poly_\varphi(\delta^{-1})$ such that the following holds.

    Let $V\subseteq M^d$ be finite. Then there is a partition $V=V_1\sqcup\cdots \sqcup V_K$ such that
    \[\sum_{V_{i_1}\timesdots V_{i_k}\text{ }\varphi\text{-homogeneous}}|V_{i_1}\timesdots V_{i_k}|\geq (1-\delta) |V|^k.\]
\end{theorem}
\patchcmd{\thmhead}{#3}{(#3)}{}{}
Observe, however, that the parts of the partition in Theorem \ref{abridgeddistalhypergraphreg} are boxes. We had previously laboured the point that, in a regularity lemma for hypergraphs, the parts of the partition generally take the shape of simplicial complexes. Thus, one expects there to be a more general setting than distality for hypergraphs that admit a homogeneous regularity lemma, in which the parts of the partition are simplicial complexes (not necessarily boxes).

We show that this is indeed the case. Where would we look for such a setting? Since distality can be seen as a binary notion (for instance, as it is a refinement of NIP\footnote{Of course, relations of arbitrary arity can be defined in an NIP theory, but the fact that such a relation is NIP only makes sense after partitioning its variables into two parts, effectively treating the relation as binary.}), it is not so surprising that while distality provided the right setting for homogeneous \textit{graph} regularity lemmas in Theorem \ref{abridgeddistalgraphreg}, it is too restrictive for the \textit{hypergraph} analogue. It is natural to look to $k$-ary generalisations of distality to provide a more general setting for homogeneous regulairty lemmas for $k$-uniform hypergraphs. This approach is not new to the literature. For example, the NIP regularity lemma for graphs \cite{lovaszszegedy, alonfischernewman} has been extended to an $\text{NIP}_k$ regularity lemma for $(k+1)$-uniform hypergraphs, where $\text{NIP}_k$ is a $(k+1)$-ary generalisation of NIP; this was achieved qualitatively by Chernikov and Towsner \cite{chernikovtowsner}, and quantitatively in the $k=2$ case by Terry and Wolf \cite{terrywolf}, Terry \cite{terryimprovedbound, terrypart3, terrypart4}, and Gishboliner, Shapira, and Wigderson \cite{gishboliner}.\footnote{Here, `qualitatively' (respectively, `quantitatively') refers to the absence (respectively, presence) of bounds obtained for the sizes of the partitions.} Another example concerns \textit{stability}, a binary notion which refines NIP: in \cite{terrywolf}, Terry and Wolf extended the stable regularity lemma for graphs \cite{malliarisshelah} to an $\text{NFOP}_2$ regularity lemma for 3-uniform hypergraphs, where $\text{NFOP}_2$ is a ternary generalisation of stability.

Returning to $k$-ary generalisations of distality, such notions were defined by Walker in \cite{walker}. Walker introduced the notions of \textit{$k$-distality} and \textit{strong $k$-distality} as $(k+1)$-ary generalisations of distality for each $k\in\N^+$, such that strong $k$-distality implies $k$-distality, (strong) $k$-distality implies (strong) $(k+1)$-distality, and 1-distality, strong 1-distality, and distality are equivalent. Our work shows that strong $(k-1)$-distality plus NIP provides the setting we desire for $k$-uniform hypergraphs with a homogeneous regularity lemma. The following is a special case of Corollary \ref{kSHDregcorfinite}.

\begin{theorem}\label{veryabridgedmyhyperreg}
    Let $k\geq 2$. Let $T$ be an NIP, strongly $(k-1)$-distal $L$-theory, $M\models T$, and let $\varphi(x_1, ..., x_k)\in L(M)$ with $|x_1|=\cdots =|x_k|=:d$. Then, for all $\delta>0$, there is a natural number $K\leq \poly_\varphi(\delta^{-1})$ such that the following holds.

    For all finite $V \subseteq M^d$, there is a partition $V^{k-1}=V_1 \sqcup \cdots \sqcup V_K$ inducing the partition $\mathcal{Q}$ of $V^k$ into simplicial complexes given by
    \[\left\lbrace\left\lbrace (v_1, ..., v_k)\in V^k: v_{\neq i}\in V_{j_i}\text{ for all } i\in [k]\right\rbrace: j_1, ..., j_k\in [K]\right\rbrace,\]
    such that $\sum_{Q\in\mathcal{Q}\text{ }\varphi\text{-homogeneous}} |Q|\geq (1-\delta) |V|^k$.
\end{theorem}

The simplicial complexes here are known as \textit{cylinder (intersection) sets} in \cite{artemfrancis, chernikovtowsner, chernikovtowsnerstability} --- see the end of Section \ref{sechypreg} for an explanation of this name. Note that this notion of a cylinder is distinct from that in a \textit{cylinder regularity lemma}, such as in \cite{cylinderregularity}.
\begin{remark}
    The reader that is familiar with general hypergraph regularity may protest that, in order for Theorem \ref{veryabridgedmyhyperreg} to be an honest regularity lemma, the parts $V_j$ of the partition of $V^{k-1}$ themselves ought to be quasirandom as $(k-1)$-uniform hypergraphs. This is not difficult to achieve. Indeed, we will show in fact that the $V_j$ are uniformly definable by some formula in the NIP theory $T$ (see, for example, Theorem \ref{nonpartitemainresult}), so we may apply the NIP regularity lemma from \cite{alonfischernewman, lovaszszegedy, definablenipregularity} to the $V_j$ to obtain a partition of $V$ into polynomially many parts $W_1\sqcup \cdots \sqcup W_H$, such that for most $j\in [K]$ and $i_1, ..., i_k\in [H]$, $V_j$ has density close to 0 or 1 inside $W_{i_1}\timesdots W_{i_k}$. See Subsection \ref{subseccounting} for more details.
\end{remark}
Theorem \ref{veryabridgedmyhyperreg} says that NIP strong $(k-1)$-distality provides a general setting for $k$-uniform hypergraphs to have a homogeneous regularity lemma. In fact, we can generalise the setting even further. To describe this generalisation, we discuss our proof strategy for Theorem \ref{veryabridgedmyhyperreg}, which involves developing a key tool for NIP strongly $k$-distal theories: \textit{$k$-strong honest definitions}.

The distal regularity lemma, Theorem \ref{abridgeddistalgraphreg}, was proved in \cite{regularitylemma} using \textit{strong honest definitions}. In \cite{extdefinable}, Chernikov and Simon proved that an $L$-theory is (1-)distal if and only if every formula $\varphi(x;y)\in L$ has a \textit{strong honest definition} $\psi(x;z)\in L$: for all $B\subseteq M\models T$ with $2\leq |B|<\infty$ and $a\in M$, there is $c\in B$ such that for all $b\in B$,
    \[a\models \psi(x;c)\vdash \varphi(x;b)\leftrightarrow \varphi(a;b).\]
Here, we have employed the abuse of notation that if $y$ is an $n$-tuple with entries in a set $Y$ (that is, $y\in Y^n$), we sometimes simply write $y\in Y$, but $X\subseteq Y$ always means $X\subseteq Y^1$. 

In Definition \ref{defnkSHDlegit}, we extend strong honest definitions to the $k$-ary setting as follows. Let $k\in\N^+$. Given an $L$-formula $\varphi(x_1, ..., x_k;y)$, a \textit{$k$-strong honest definition} for $\varphi$ is a $(k+1)$-tuple of $L$-formulas $(\psi_i(x_{\neq i},y,z_i):i\in [k])^\frown (\psi_{k+1}(x,z_{k+1}))$ such that the following holds.

    There is $N\in\N$ such that, for all $B\subseteq M\models T$ with $2\leq |B|<\infty$ and $a=(a_1, ..., a_k)\in M$, there are $c^{(j)}_1, ..., c^{(j)}_{k+1}\in B$ for $j\in [N]$ such that for all $b\in B$, there is $j\in [N]$ such that
            \[a\models \psi_{k+1}(x,c^{(j)}_{k+1})\wedge\bigwedge_{i=1}^k \psi_i(x_{\neq i},b,c^{(j)}_i)\vdash \varphi(x;b)\leftrightarrow \varphi(a;b).\]

The following analogue to the result of Chernikov and Simon is Theorem \ref{uniformkSHD}.
\begin{theorem}\label{introequiv}
    Let $T$ be NIP and $k\in\N^+$. Then $T$ is strongly $k$-distal if and only if every $\varphi(x_1, ..., x_k; y)\in L$ has a $k$-strong honest definition.
\end{theorem}

The reader may be concerned that the formulas $\psi_1, ..., \psi_k$ in a $k$-strong honest definition for $\varphi(x_1, ..., x_k; y)$ involve the $y$-variable as well as the $x$-variables. Crucially, however, each of $\psi_1, ..., \psi_k$ involves exactly $k-1$ of the $x$-variables (and $y$). The intuition is that, for $b\in M$, in order to understand how $x_1, ..., x_k$ interact with $b$ (with respect to $\varphi$), it is enough to understand how any $k-1$ of the $x_i$'s interacts with $b$ (with respect to $\psi_1, ..., \psi_k$) and how $x_1, ..., x_k$ interact (with respect to $\psi_{k+1}$). In other words, the interaction of the $k+1$ variables $x_1, ..., x_k, y$ is locally controlled by the interactions of $k$ of those variables.

We now come to the main result of the paper. The reader can find it in full generality and strength as Theorem \ref{kSHDreglemmauniform}, but here we first state an abridged form that fits more directly with the narrative so far. This is a special case of Corollary \ref{kSHDregcorfinite}.

\begin{theorem}\label{abridgedmyhyperreg}
    Let $k\geq 2$. Let $M$ be an $L$-structure that is NIP, and let $\varphi(x_1, ..., x_{k-1}; x_k)\in L(M)$ have a $(k-1)$-strong honest definition, with $|x_1|=\cdots =|x_k|=:d$. Then, for all $\delta>0$, there is a natural number $K\leq \poly_\varphi(\delta^{-1})$ such that the following holds.

    For all finite $V \subseteq M^d$, there is a partition $V^{k-1}=V_1 \sqcup \cdots \sqcup V_K$ inducing the partition $\mathcal{Q}$ of $V^k$ given by
    \[\left\lbrace\left\lbrace (v_1, ..., v_k)\in V^k: v_{\neq i}\in V_{j_i}\text{ for all } i\in [k]\right\rbrace: j_1, ..., j_k\in [K]\right\rbrace,\]
    such that $\sum_{Q\in\mathcal{Q}\text{ }\varphi\text{-homogeneous}} |Q|\geq (1-\delta) |V|^k$.
\end{theorem}

The assumption is weaker here than in Theorem \ref{veryabridgedmyhyperreg}. Indeed, here we only require that $\varphi$ has a $(k-1)$-strong honest definition, whereas in Theorem \ref{veryabridgedmyhyperreg} we require all formulas to have a $(k-1)$-strong honest definition (by Theorem \ref{introequiv}).

We now state the main result in greater generality and strength, in the language of Keisler measures (note that Theorems \ref{abridgeddistalgraphreg} and \ref{abridgeddistalhypergraphreg} can also be formulated in terms of Keisler measures). This is a special case of Corollary \ref{kSHDregcor}.
\begin{theorem}\label{nonpartitemainresult}
    Let $k\geq 2$. Let $M$ be an $L$-structure that is NIP, and let $\varphi(x_1, ..., x_{k-1}; x_k)\in L(M)$ have a $(k-1)$-strong honest definition, with $|x_1|=\cdots =|x_k|=:d$. Then, for all $\delta>0$, there is a natural number $K\leq \poly_\varphi(\delta^{-1})$ and a formula $\theta(x_1, ..., x_{k-1},z)\in L$ such that the following holds.
    
    Let $V\subseteq M^d$ be $M$-definable, and let $\mu(x_1)$ be a global measure, generically stable over $M$. Then there is a partition $V^{k-1}=V_1\sqcup\cdots \sqcup V_K$, where each $V_i=\theta(x_1, ..., x_{k-1}, c)$ for some $c\in M^z$, inducing the partition
    \[\mathcal{Q}:=\left\lbrace\left\lbrace v=(v_1, ..., v_k)\in V^k: v_{\neq i}\in V_{j_i}\text{ for all } i\in [k]\right\rbrace: j_1, ..., j_k\in [K]\right\rbrace\]
    of $V^k$, such that $\sum_{Q\in\mathcal{Q}\text{ }\varphi\text{-homogeneous}} \mu^{(k)}(Q)\geq (1-\delta) \mu(V)^k$.
\end{theorem}
Note that although we develop $k$-strong honest definitions as a tool to find our hypergraph homogeneous regularity lemma, we believe their potential applicability far exceeds the context of this paper. Just as strong honest definitions proved to be instrumental in the development of distality (see, for instance, \cite{distalityvaluedfields, boxallkestner, cuttinglemma, regularitylemma}), we hope that our $k$-strong honest definitions will similarly lead to further developments in higher-arity distality.
\subsection{Similar results in the literature}
Most of the results in this paper, including the regularity lemma and the development of $k$-strong honest definitions, are contained in the author's PhD thesis \cite{thesis}.

In \cite[Corollary 3.33]{artemfrancis}, Chernikov and Westhead also obtain a homogeneous regularity lemma for formulas in NIP strongly $k$-distal structures, using completely different methods: they extend the notion of smooth measures to that of \textit{$k$-smooth} measures, and prove that $k$-smooth measures have good properties in NIP strongly $k$-distal theories (analogous to those of smooth measures in distal theories).

Unlike \cite{artemfrancis}, we obtain a bound on the size of the partition, namely, that it is polynomial in the reciprocal of the error parameter. Furthermore, we prove that the parts of the partition are uniformly definable; in \cite[Corollary 3.33]{artemfrancis}, this is proven under an additional assumption, such as the existence of definable Skolem functions. This assumption is also required in their regularity lemma for arbitrary generically stable measures; without this assumption, their result is stated for finitely supported measures. Here, we prove the regularity lemma for arbitrary generically stable measures without additional assumptions.
\subsection{Structure of the paper}
We begin with expositions of hypergraph regularity lemmas and higher-arity distality, constituting Sections \ref{sechypreg} and \ref{sechigherarity} respectively. In Section \ref{secSHD}, we define $k$-strong honest definitions, and show that an NIP theory is strongly $k$-distal if and only if every formula $\varphi(x_1, ..., x_k; y)$ has a $k$-strong honest definition. In Section \ref{seckdistalreglemma}, we state and prove our regularity lemma for formulas $\varphi(x_1, ..., x_k; y)$ with a $k$-strong honest definition in an NIP theory.
\subsection{Basic definitions}\label{subsecnotation}
We lay out some basic notation used in the rest of this paper.

Let $k,l\in\N^+$. A \textit{$k$-uniform hypergraph}, or a \textit{$k$-graph}, is a pair $H=(V,E)$ where $V$ is the set of \textit{vertices} and $E\subseteq \binom{V}{k}$ is the set of \textit{hyperedges}, that is, $E$ consists of subsets of $V$ of size $k$. We sometimes consider hyperedges as tuples rather than sets. If the hyperedge relation $E$ is not specified, we sometimes denote it by $H$.

A $k$-graph $H=(V,E)$ is \textit{$l$-partite} if there is a partition $V_1\sqcup \cdots \sqcup V_l$ of $V$ such that, for all $e\in E$ and $i\in [l]$, $|e\cap V_i|\leq 1$; in this case, we write $H=(V_1\sqcup \cdots \sqcup V_l, E)$, and if $k=l$, we write $H=E(V_1, ..., V_k)$ and view it as a subset of $V_1\timesdots V_k$. We sometimes define a $k$-partite $k$-graph $E(V_1, ..., V_k)$ where $V_1, ..., V_k$ are not necessarily disjoint; in that case, the vertex sets are taken to be disjoint copies of $V_1, ..., V_k$.

For $q,r\in\R$ and $\delta\geq 0$, write $q\approx_\delta r$ to mean $|q-r|\leq \delta$.

We use the following asymptotic notation. Let $f(x,y),g(x,y): D\times E\to\R_{\geq 0}$, where $D,E$ are sets and $x,y$ are tuples of variables.
\begin{enumerate}[(i)]
    \item Write $f(x,y)=O_x(g(x,y))$ if there is $C=C(x): D\to \R_{\geq 0}$ such that $f(x,y)\leq C g(x,y)$ for all $x\in D$ and $y\in E$.
    \item Let $h(y): E\to\R_{\geq 0}$. Write $f(x,y)\leq \text{poly}_x(h(y))$ if there is $C=C(x): D\to\R_{\geq 0}$ such that $f(x,y)\leq Ch(y)^C$ for all $x\in D$ and $y\in E$.
\end{enumerate}

We often work in a `sufficiently saturated' model $\monster$ of a given first-order theory, in which case a subset $B\subseteq \monster$ is \textit{small} if $|B|$ is strictly less than the saturation cardinality of $\monster$.
\subsection{Acknowledgements}
I thank Artem Chernikov, Aris Papadopoulos, and Francis Westhead for fruitful conversations on $k$-distality at the University of Maryland, with special thanks to Artem for his invitation. I also thank Julia Wolf for her patient explanations of hypergraph regularity and for her feedback on this paper. Finally, I thank Pantelis Eleftheriou for his support and mentorship throughout my PhD.

\textit{Soli Deo gloria.}
\subsection{Funding statement}
The majority of the research in this paper was conducted as part of the author’s PhD, supported by a scholarship funded by the School of Mathematics at the University of Leeds. This manuscript was prepared while the author was employed at the University of Cambridge and supported by Julia Wolf's Open Fellowship from the UK Engineering and Physical Sciences Research Council (EP/Z53352X/1).
\section{Hypergraph regularity}\label{sechypreg}
This section is an exposition of hypergraph regularity lemmas. There are many excellent such expositions in the literature (the reader is recommended to peruse \cite{gowershypergraphs}, for instance), but our aim is to give an abridged account of the subject matter, highlighting the features required to motivate the main result that is our own hypergraph regularity lemma (Theorem \ref{kSHDreglemmauniform}).

In this section, all (hyper)graphs are finite. We begin with a definition used to state Szemer\'edi's regularity lemma.
\begin{defn}\label{defndeltaregular}
    Let $E(V_1, V_2)$ be a bipartite graph. For $W\subseteq V_1\times V_2$, the \textit{relative density} of $E(V_1, V_2)$ in $W$ is
    \[d_W(V_1, V_2):=\frac{|E(V_1, V_2)\cap W|}{|W|}.\]
    The \textit{density} of $E(V_1, V_2)$ is $d(V_1, V_2):=d_{V_1\times V_2}(V_1, V_2)$.
    
    For $\delta>0$, say that $V_1\times V_2$ (or $(V_1, V_2)$) is \textit{$\delta$-regular} if, for all $A_i\subseteq V_i$ with $|A_i|\geq \delta |V_i|$, $d_{A_1\times A_2}(V_1, V_2)\approx_\delta d(V_1, V_2)$.
\end{defn}
This is a notion of \textit{quasirandomness}: if $E(V_1,V_2)$ is $\delta$-regular, then any pair of vertex subsets induces (approximately) the expected number of edges, which is a property of a random graph.

Szemer\'edi's regularity lemma reads as follows.
\patchcmd{\thmhead}{(#3)}{#3}{}{}
\begin{theorem}[\cite{szemerediregularitylemma}]\label{szemeredigraph}
    For all $\delta>0$, there is $K\in\N$ such that the following holds.

    Let $G=(V\sqcup V,E)$ be a finite bipartite graph. Then there is a partition $V=V_1\sqcup \cdots \sqcup V_K$ such that
    \[\sum_{V_i\times V_j\text{ }\delta\text{-regular}}|V_i\times V_j|\geq (1-\delta) |V|^2.\]
\end{theorem}
\patchcmd{\thmhead}{#3}{(#3)}{}{}
\begin{remark}\label{bipartiteornot}
    The reader may be used to a statement of Szemer\'edi's regularity lemma that is for a general graph rather than a bipartite graph, but these statements are essentially equivalent. Indeed, the bipartite version follows immediately from the general version; conversely, given a graph $G=(V,E)$, apply the bipartite version to the auxiliary bipartite graph $(V\sqcup V, E')$, where $(x,y)\in E'$ if and only if $(x,y)\in E$.

    In this paper, we will often focus on the partite setting, as it makes some statements cleaner and connects better to model-theoretic relations. In view of the above, the reader is urged not to worry about this.
\end{remark}
A notable application of Szemer\'edi's regularity lemmas is \textit{(graph) removal lemmas}, of which the following \textit{triangle removal lemma} is an important example. The \textit{triangle} is the complete graph on 3 vertices, and a graph is \textit{triangle-free} if it has no triangles.
\patchcmd{\thmhead}{(#3)}{#3}{}{}
\begin{theorem}[\cite{triangleremoval}]\label{removallemma}
    For all $\gamma>0$, there is $\alpha>0$ such that the following holds. If $G$ is a graph on $n$ vertices with fewer than $\alpha n^3$ triangles, then it can be made triangle-free by removing at most $\gamma n^2$ edges.
\end{theorem}
\patchcmd{\thmhead}{#3}{(#3)}{}{}

In order to deduce this from Szemer\'edi's regularity lemma, one needs a \textit{triangle counting lemma}. For a precise statement of this, the reader is referred to \cite[Theorem 1.2]{gowersregularity}. Roughly speaking, this says that given a tripartite graph $G$ on vertex sets $V_1$, $V_2$, $V_3$, if each of the three pairs of vertex sets is $\delta$-regular, then $G$ contains approximately (in terms of $\delta$) the correct number of triangles, namely, the number of triangles $G$ would contain if its edges were chosen uniformly at random. This is to be expected for a good notion of quasirandomness.

How does one generalise the above to hypergraphs? Perhaps the most obvious approach is to make the following definition.
    \begin{defn}\label{defndeltaregularhyper}
        Let $E(V_1, ..., V_k)$ be a $k$-partite $k$-graph. Let $W\subseteq V_1\timesdots V_k$. The \textit{relative density} of $E(V_1, ..., V_k)$ in $W$ is
    \[d_W(V_1, ..., V_k):=\frac{|E(V_1, ..., V_k)\cap W|}{|W|}.\]
    The \textit{density} of $E(V_1, ..., V_k)$ is $d(V_1, ..., V_k):=d_{V_1\timesdots V_k}(V_1, ..., V_k)$.

        For $\delta>0$, say that $V_1\timesdots V_k$ (or $(V_1, ..., V_k)$) is \textit{$\delta$-regular} if, for all $A_i\subseteq V_i$ with $|A_i|\geq \delta|V_i|$, $d_{A_1\timesdots A_k}(V_1, ..., V_k)\approx_\delta d(V_1, ..., V_k)$.
    \end{defn}
    With this definition, there is indeed a version of Szemer\'edi's regularity lemma for hypergraphs, due to Chung.
    \patchcmd{\thmhead}{(#3)}{#3}{}{}
    \begin{theorem}[\cite{chungregularity}]\label{chung}
    For all $k\in\N^+$ and $\delta>0$, there is $K\in\N$ such that the following holds.

    Let $H=(\bigsqcup_{i=1}^k V,E)$ be a $k$-partite $k$-graph. Then there is a partition $V=V_1\sqcup \cdots \sqcup V_K$ such that
    \[\sum_{V_{i_1}\timesdots V_{i_k}\text{ }\delta\text{-regular}}|V_{i_1}\timesdots V_{i_k}|\geq (1-\delta) |V|^k.\]
\end{theorem}
\patchcmd{\thmhead}{#3}{(#3)}{}{}
We would likewise want a $k$-graph version of the triangle removal lemma. The natural generalisation of the triangle is the $k$-simplex: the complete $k$-graph on $k+1$ vertices. Unfortunately, such a result cannot be deduced from Theorem \ref{chung}, because there is no $k$-simplex counting lemma for $\delta$-regular hypergraphs: $\delta$-regularity is not strong enough for counting $k$-simplices when $k\geq 3$. This fact is illustrated in \cite{linearhypergraphs} by the following example (which we have adapted slightly), but the authors of \cite{linearhypergraphs} describe the fact as `known'.
\begin{example}[$\delta$-regularity insufficient for counting tetrahedra]\label{deltaregcantcount}
    Let $V_1$, $V_2$, $V_3$, $V_4$ be sets of size $n$. Let $H_1$ be the random 4-partite 3-graph on $V_1\sqcup\cdots\sqcup V_4$, where hyperedges occur with probability $1/8$. Let $G_2$ be the random 4-partite graph on $V_1\sqcup\cdots\sqcup V_4$, where edges occur with probability $1/2$, and let $H_2$ be the 4-partite 3-graph on $V_1\sqcup\cdots\sqcup V_4$ whose hyperedges are precisely the triangles in $G$. Then, with high probability, for all $\delta>0$, every triple of vertex sets in both $H_1$ and $H_2$ is $\delta$-regular with density $\approx_\delta 1/8$. However, $H_1$ (which is `truly random') is expected to contain $(1/8^4)n^4$ tetrahedra, while $H_2$ is expected to contain $(1/2^6)n^4$ tetrahedra.
\end{example}

In fact, $\delta$-regularity is insufficient even to guarantee the \textit{existence} of tetrahedra --- see \cite{linearhypergraphs}.

Thus, $\delta$-regularity is not a good notion of quasirandomness for $k$-graphs when $k\geq 3$. What, then, is the correct notion? As mentioned in Section \ref{secintro}, this question was independently addressed by several groups of authors. We will give a much abridged version of Gowers' accounts in \cite{gowershypergraphs} and \cite{gowersregularity}. The reader should bear in mind that, although we shall attribute the results to Gowers, similar results have been obtained by others --- in particular, by Nagle, R\"odl, Schacht, and Skokan \cite{nrs, rodlskokan} --- and is referred to \cite{regularityhistory} for a more complete historical account.

As alluded to in Section \ref{secintro}, the problem is not so much that we need a better notion of quasirandomness per se, but that the parts of our partition have the wrong shape. We have been considering what it means for a hypergraph $H=(V,E)$ to be quasirandom inside a box $V_1\timesdots V_k$, where $V_i\subseteq V$, but it is time to move beyond boxes.

Observe that a partition of $V^k$ into boxes is induced by a partition $\mathcal{P}$ of the vertex set $V$: each part of the partition of $V^k$ is obtained by choosing $P_1, ..., P_k\in\mathcal{P}$ and forming the box
\[\{(v_1, ..., v_k)\in V^k: v_j\in P_j\text{ for all }j\in [k]\}.\]

It turns out that we should partition not only the vertex set $V$ but $V^1, ..., V^{k-1}$. Naming the partitions $\mathcal{P}_1, ..., \mathcal{P}_{k-1}$, we can partition $V^k$ as follows. Each part of the partition is obtained by choosing $P_I\in \mathcal{P}_{|I|}$ for each $\emptyset\neq I\subsetneq [k]$ and forming the `simplicial complex'
\[\{(v_1, ..., v_k)\in V^k: v_I\in P_I\text{ for all }\emptyset\neq I\subsetneq [k]\}.\]
Note that if the partitions of $V^2, ..., V^{k-1}$ are trivial, the resulting simplicial complexes are boxes. Thus, a partition into boxes is a special case of a partition into simplicial complexes.

Why the name of `simplicial complexes'? Recall that a $(k-1)$-dimensional \textit{(abstract) simplicial complex} is a set $\Sigma$ of sets of size at most $k$, such that if $B\in \Sigma$ and $A\subseteq B$ then $A\in \Sigma$. Hence, we are really saying that each part of the partition of $V^k$ encodes the \textit{data} of a $(k-1)$-dimensional simplicial complex. Each part of the partition consists of $k$-tuples, which are in some sense generated by collections of $l$-tuples for $l\in [k-1]$.

The $k$-graph generalisation of Szemer\'edi's regularity lemma now says that, given a finite $k$-partite $k$-graph $H=(\bigsqcup_{i=1}^k V, E)$ and an error parameter $\delta>0$, there are partitions of $V^1, ..., V^{k-1}$ whose sizes are bounded in terms of $\delta$, inducing a partition of $V^k$ into simplicial complexes, such that $E$ is quasirandom with respect to most of the simplicial complexes. The precise notion of quasirandomness, which we call \textit{$\delta$-quasirandomness} here, is rather involved and is outside the scope of this paper; we refer the reader to \cite{gowersregularity} (or \cite{gowershypergraphs} for the $k=3$ case) for an exposition.

It is nonetheless worth highlighting one key feature of $\delta$-quasirandomness; we specialise to the $k=3$ case for this. In this case, we have partitions $\mathcal{P}_1$ of $V^1$ and $\mathcal{P}_2$ of $V^2$, and each simplicial complex $Q$ has the form
\[Q=\{(v_1, v_2, v_3)\in V^k: v_i\in V_i, (v_i, v_j)\in G_{ij}\text{ for all }i,j\}\]
for some $V_1, V_2, V_3\in \mathcal{P}_1$ and $G_{12}, G_{13}, G_{23}\in\mathcal{P}_2$. Since the $G_{ij}$ are subsets of $V^2$, we may view them as bipartite graphs on $V\sqcup V$. In this light, $Q$ is precisely the set of triangles in the graph $G_{12}(V_1, V_2)\cup G_{13}(V_1, V_3)\cup G_{23}(V_2, V_3)$.

It turns out that, in order to define $\delta$-quasirandomness, one should specify not only that the hypergraph relation is quasirandom with respect to the simplicial complex as a whole, but also that the simplicial complex itself is quasirandom at each level: namely, that the graphs $G_{ij}(V_i, V_j)$ should also be quasirandom. This notion of quasirandomness is none other than $\varepsilon$-regularity (for some $\varepsilon=\varepsilon(\delta)$), since we have established that this is a good notion of quasirandomness for graphs.

We now state a version of Gowers' hypergraph regularity lemma in \cite{gowersregularity}, which --- having avoided a definition of $\delta$-quasirandomness --- is necessarily much abridged, but is sufficient for motivating the rest of this paper.
\patchcmd{\thmhead}{(#3)}{#3}{}{}
\begin{theorem}[\cite{gowersregularity}]\label{gowersregularity}
    For all $k\in\N^+$ and $\delta>0$, there is $K\in\N$ such that the following holds.

    Let $H=(\bigsqcup_{i=1}^k V,E)$ be a $k$-partite $k$-graph. Then there are partitions of $V^1, ..., V^{k-1}$ of size at most $K$, inducing a partition $\mathcal{Q}$ of $V^k$ into $(k-1)$-dimensional simplicial complexes, such that
    \[\sum_{\substack{Q\in\mathcal{Q}\\Q\text{ }\delta\text{-quasirandom}}}|Q|\geq (1-\delta) |V|^k.\]
\end{theorem}
\patchcmd{\thmhead}{#3}{(#3)}{}{}
Associated to this notion of quasirandomness, Gowers proves a $k$-simplex counting lemma, which can be combined with Theorem \ref{gowersregularity} to give the desired removal lemma.
\patchcmd{\thmhead}{(#3)}{#3}{}{}
\begin{theorem}[\cite{gowersregularity}]\label{simplexremovallemma}
    For all $\gamma >0$, there is $\alpha >0$ such that the following holds. If $G$ is a $k$-uniform hypergraph on $n$ vertices with fewer than $\alpha n^{k+1}$ $k$-simplices, then it can be made $k$-simplex-free by removing at most $\gamma n^k$ hyperedges.
\end{theorem}
    We reassure the reader concerned over the precise nature of $\delta$-quasirandomness that the remainder of this paper is comprehensible with this definition as a black box. Indeed, recall that our goal is to find out when a hypergraph $E$ may be decomposed into simplicial complexes $Q$ that are \textit{($E$-)homogeneous}, that is, $E\cap Q=Q$ or $\emptyset$, and it suffices to note that homogeneity is a (very strong) special case of $\delta$-quasirandomness.

    In fact, since we will be working with the special case of homogeneity, we can simplify our construction of simplicial complexes. We illustrate this by first specialising to 3-graphs. Recall that in Gowers' regularity lemma for a 3-partite 3-graph $(\bigsqcup_{i=1}^3 V,E)$, there are partitions $\mathcal{P}_1,\mathcal{P}_2$ of $V^1, V^2$ which induce a partition of $V^3$ into simplicial complexes. Each simplicial complex has the form
    \[Q:=\{(v_1, v_2, v_3)\in V^3: v_i\in P_i\text{ for all }i, (v_i, v_j)\in P_{ij}\text{ for all }i,j\},\]
where $P_1, P_2, P_3\in\mathcal{P}_1$ and $P_{12}, P_{23}, P_{13}\in \mathcal{P}_2$. Now, for each $i,j$, let $P'_{ij}:=(P_i\times P_j)\cap P_{ij}$. We may rewrite $Q$ as
\begin{equation}\label{eqnnewpartition}
    \{(v_1, v_2, v_3)\in V^3: (v_i, v_j)\in P'_{ij}\text{ for all }i,j\}.
\end{equation}
In this way, we may combine $\mathcal{P}_1$ and $\mathcal{P}_2$ into a single partition $\mathcal{P}'$ of $V^2$, such that each simplicial complex is of the form (\ref{eqnnewpartition}) for some $P'_{12}, P'_{23}, P'_{13}\in \mathcal{P}'$.

For the purposes of Gowers' regularity lemma, this manoeuvre would be misguided: recall that the notion of $\delta$-quasirandomness depends on the partitions of $V^1$ and $V^2$ separately. However, the notion of homogeneity only depends on the simplicial complex $Q$ as a whole: $Q$ is homogeneous if the restriction of the hypergraph to $Q$ is complete or empty. Thus, when working with homogeneity, we can combine the partitions of $V^1$ and $V^2$ as above.

Summarising and extrapolating to $k$-graphs $(\bigsqcup_{i=1}^k V,E)$, to find a partition of $V^k$ into $E$-homogeneous simplicial complexes, it suffices to consider those induced by a partition $\mathcal{P}$ of $V^{k-1}$; the parts of the partition (simplicial complexes) then have the form
\[\{ (v_1, ..., v_k)\in V^k: v_{\neq i}\in P_i\text{ for all } i\},\]
where $P_1, ..., P_k\in\mathcal{P}$. A set of this form is known as a \textit{cylinder (intersection) set}: given $P_1, ..., P_k\in\mathcal{P}$, for all $i\in [k]$, the set $\{(v_1, ..., v_k)\in V^k: v_{\neq i}\in P_i\}$ can be thought of as `the cylinder above $P_i$' --- it is the set of points whose appropriate projection lies in $P_i$ --- and the set in question is the intersection of these cylinders.
\patchcmd{\thmhead}{#3}{(#3)}{}{}
\section{Higher-arity distality}\label{sechigherarity}
In this section, we state the definitions of $k$-distality and strong $k$-distality from \cite{walker}, and give some basic properties and examples. Throughout this section, fix a complete $\ELL$-theory $T$, and let $\monster\models T$ be sufficiently saturated. Unless otherwise specified, we only work with subsets of $\monster$ that are small.

To motivate these definitions of higher-arity distality, let us recall some facts about distality. The following was Simon's original definition of distality in \cite{simondistal}.
\begin{defn}
    Say that $T$ (and any $M\models T$) is \textit{distal} if the following holds.

    Let $I_0, I_1, I_2$ be infinite sequences of tuples from $\monster$. Let $a_1, a_2$ be tuples from $\monster$. If $I_0+I_1+a_2+I_2$ and $I_0+a_1+I_1+I_2$ are indiscernible, then $I_0+a_1+I_1+a_2+I_2$ is indiscernible.
\end{defn}

It is almost folklore that the previous definition of distality implies NIP, but to our knowledge, a proof only appears in the literature in \cite[Corollary 6.8]{walker}, due to Chernikov. In particular, Simon's paper \cite{simondistal} uses NIP as an extra assumption rather than as a consequence of distality.

In \cite{simondistal}, Simon proves (modulo this \textit{a posteriori} superfluous NIP assumption) that the following is an equivalent characterisation of distality.

\begin{defn}[External characterisation of distality]
    Say that $T$ (and any $M\models T$) is \textit{distal} if the following holds.

    Let $I, J$ be infinite sequences of tuples from $\monster$. Let $a,b$ be tuples from $\monster$. If $I+a+J$ is indiscernible and $I+J$ is indiscernible over $b$, then $I+a+J$ is indiscernible over $b$.
\end{defn}
Walker \cite{walker} generalised these two definitions as follows.
\begin{defn}\label{defnkdistal}
    Let $k\in\N^+$. Say that $T$ (and any $M\models T$) is \textit{$k$-distal} if the following holds.

    Let $I_0, ..., I_{k+1}$ be infinite sequences of tuples from $\monster$. Let $a_1, ..., a_{k+1}$ be tuples from $\monster$. If $I_0+a_1+I_1+a_2+\cdots+I_{m-1}+I_m+\cdots+a_{k+1}+I_{k+1}$ is indiscernible for all $m\in [k+1]$, then $I_0+a_1+I_1+a_2+\cdots+I_k+a_{k+1}+I_{k+1}$ is indiscernible.
\end{defn}
\begin{defn}\label{defnkstronglydistal}
    Let $k\in\N^+$. Say that $T$ (and any $M\models T$) is \textit{strongly $k$-distal} if the following holds.

    Let $I, J$ be infinite sequences of tuples from $\monster$. Let $a, b_1, ..., b_k$ be tuples from $\monster$. If $I+a+J$ is indiscernible over $b_{\neq m}$ for all $m\in [k]$ and $I+J$ is indiscernible over $b_1\cdots b_k$, then $I+a+J$ is indiscernible over $b_1\cdots b_k$.
\end{defn}
Here, $b_{\neq m}:=\{b_i: i\in [k]\setminus \{m\}\}$ for all $m\in [k]$. It is straightforward to see that (strong) $k$-distality implies (strong) $(k+1)$-distality, and that strong $k$-distality implies $k$-distality. The converse to the latter statement does not hold \cite[Corollary 4.23]{artemfrancis}.

Note that both $k$-distality and strong $k$-distality say that the definable interaction of $k+1$ variables can be controlled by the definable interactions of $k$-sized subsets of those variables. Indeed, $k$-distality says that if any $k$ of $a_1, ..., a_{k+1}$ can be inserted to make an indiscernible sequence, then all $k+1$ of them can be, and strong $k$-distality says that if $I+J$ is indiscernible with respect to any $k$ of $a, b_1, ..., b_k$, then it is indiscernible with respect to all $k+1$ of them.

Let us give some examples of (strongly) $k$-distal theories. Since we are interested in higher-arity distality, we will focus on non-distal examples; a list of examples of distal theories can be found in \cite[Section 2.3]{regularitylemma}. We begin by describing a somewhat degenerate source of (strongly) $k$-distal theories. For $k\in\N^+$, say that $T$ (and any $M\models T$) is \textit{$k$-ary} if every relation is equivalent in $T$ to a Boolean combination of $k$-ary relations; equivalently, if $T$ admits quantifier elimination in a $k$-ary relational language.
\patchcmd{\thmhead}{(#3)}{#3}{}{}
\begin{fact}[{\cite[Corollary 4.14]{walker}}]\label{k-aryisstronglyk-distal}
    Let $k\in \N^+$. If $T$ is $k$-ary, then $T$ is (strongly) $k$-distal.
\end{fact}
\patchcmd{\thmhead}{#3}{(#3)}{}{}

Since (strong) $k$-distality is intended to characterise the definable interactions of $k+1$ variables, $k$-ary examples of (strongly) $k$-distal theories are not very illuminating. Indeed, in a $k$-ary theory, every $(k+1)$-ary relation is degenerate, in the sense that it is a Boolean combination of $k$-ary relations. Furthermore, our goal in this paper is to find model-theoretic contexts for homogeneous $(k+1)$-graph regularity lemmas; it is clear that $k$-ary theories are not fit for this purpose.

In \cite{trivialitypaper}, we give four classes of (strongly) $k$-distal theories that are not $k$-ary. We describe one such class here. For $k\in\N^+$, the \textit{Johnson graph} $\mathfrak{J}(k)$ is the graph with vertex set $\binom{\N}{k}$ such that, for $x,y\in \binom{\N}{k}$, $(x,y)$ is an edge if and only if $|x\cap y|=k-1$. Consider this as a structure in the language of graphs. Observe that this is interpretable in $(\N,=)$, and so is $\omega$-stable; in particular, it is not distal.
\patchcmd{\thmhead}{(#3)}{#3}{}{}
\begin{fact}[{\cite[Proposition 3.19]{trivialitypaper}}]
    The structure $\mathfrak{J}(k)$ is (strongly) 2-distal.
\end{fact}
\patchcmd{\thmhead}{#3}{(#3)}{}{}
It is easy to observe that $\mathfrak{J}(k)$ is not 2-ary whenever $k\geq 2$ (see, for example, \cite[before Proposition 3.18]{trivialitypaper}). In fact, by \cite[Proposition 3.18]{trivialitypaper}, there is an increasing unbounded function $g:\N^+\to\N^+$ such that $\mathfrak{J}(k)$ is not $g(k)$-ary for all $k\in\N^+$.

The structure $\mathfrak{J}(k)$, for $k\geq 2$, will be a helpful example to keep in mind for the rest of the paper. In particular, since we will prove a homogeneous regularity lemma for 3-graphs definable in an NIP strongly 2-distal structure, this can be applied to the relation $\varphi_l(x,y,z)$ saying that $|x\cap y\cap z|=l$, where $0\leq l\leq k-1$ is a fixed integer; this relation is definable in $\mathfrak{J}(k)$ \cite[before Proposition 3.16]{trivialitypaper}.
\section{Higher-arity strong honest definitions}\label{secSHD}
In this section, we introduce $k$-strong honest definitions for $(k+1)$-ary formulas and show that their existence characterises strong $k$-distality among NIP theories. Not only is it a key tool for the proof of our main result --- a regularity lemma for NIP strongly $k$-distal structures --- it is also a result of independent interest. Just as strong honest definitions have proved crucial in the development of distality, it is our hope that $k$-strong honest definitions will take on the same role in the development of strong $k$-distality.

Throughout this section, fix a complete $\ELL$-theory $T$, and let $\monster\models T$ be sufficiently saturated. Unless otherwise specified, we only work with subsets of $\monster$ that are small. We reiterate our abuse of notation that if $y$ is an $n$-tuple with entries in a set $Y$ (that is, $y\in Y^n$), we sometimes simply write $y\in Y$, but $X\subseteq Y$ always means $X\subseteq Y^1$.

Recall the definition of strong honest definitions for a binary formula. 
\begin{defn}
    Let $\varphi(x;y)\in L$. A formula $\psi(x;z)\in L$ is a \textit{strong honest definition} for $\varphi$ if the following holds.

    Let $B\subseteq M\models T$ with $2\leq |B|< \infty$, and let $a\in M$. Then there is $c\in B$ such that, for all $b\in B$,
    \[a\models \psi(x;c)\vdash \varphi(x;b)\leftrightarrow \varphi(a;b).\]
\end{defn}
The seminal result that is \cite[Theorem 21]{extdefinable} states that the existence of strong honest definitions characterises distality. We expand the statement slightly as follows.
\begin{theorem}\label{distalequivSHD}
    The following are equivalent.
    \begin{enumerate}[(i)]
        \item The theory $T$ is distal.
        \item Every formula $\varphi(x;y)\in L$ has a strong honest definition.
        \item Let $\varphi(x;y)\in L$, $\emptyset\neq B\subseteq M\models T$, and $a\in M$. There is $\psi(x;z)\in L$ such that, for all finite $\bar{B}\subseteq B$, there is $c\in B$ such that, for all $b\in \bar{B}$,
    \[a\models \psi(x;c)\vdash \varphi(x;b)\leftrightarrow \varphi(a;b).\]
    \end{enumerate}
\end{theorem}
\begin{proof}
    That (i) is equivalent to (ii) is, modulo a compactness argument, \cite[Theorem 21]{extdefinable}. The proof can be used almost verbatim to show that (i) is equivalent to (iii).
\end{proof}
Statement (iii) gives a `non-uniform' strong honest definition: one that depends not only on the formula $\varphi(x;y)$ but also on the parameters $a$ and $B$.

We wish to define $k$-strong honest definitions for $(k+1)$-ary formulas, where $k\in\N^+$, and use them to characterise $k$-distality or strong $k$-distality. Walker proves the following result that makes a significant step towards this goal. For a tuple $a=(a_1, ..., a_k)$ and $i\in [k]$, write $a_{\neq i}:=(a_1, ..., a_{i-1}, a_{i+1}, ..., a_k)$.
\patchcmd{\thmhead}{(#3)}{#3}{}{}
\begin{theorem}[{\cite[Theorem 9.18]{walkerthesis}}]\label{walkerSHDnew}
    Let $k\in\N^+$. The following are equivalent.
    \begin{enumerate}[(i)]
        \item The theory $T$ is strongly $k$-distal.
        \item Let $\varphi(x;y)\in L$ with $x=(x_1, ..., x_k)$, $\emptyset\neq B\subseteq M\models T$, and $a=(a_1, ..., a_k)\in M$. Then there is $\psi(x;z)\in L$ such that, for all finite $\bar{B}\subseteq B$, there is $c\in B$ such that, for all $b\in\bar{B}$,
        \begin{equation}\label{eqnwalkerSHDnew}
        a\models \{\psi(x;c)\}\cup\bigcup_{i=1}^k\tp(a_{\neq i}/B)\vdash \varphi(x;b)\leftrightarrow \varphi(a;b).
    \end{equation}
    \end{enumerate}
\end{theorem}
\patchcmd{\thmhead}{#3}{(#3)}{}{}
When $k=1$, (\ref{eqnwalkerSHDnew}) simplifies to
\[a\models \psi(x;c)\vdash \varphi(x;b)\leftrightarrow \varphi(a;b);\]
that is, $\psi$ is precisely a `non-uniform' strong honest definition for $(\varphi,a,B)$, as in statement (iii) of Theorem \ref{distalequivSHD}. It may therefore be tempting to define, for arbitrary $k$, $\psi$ to be a `non-uniform' $k$-strong honest definition for $(\varphi,a,B)$.

This turns out to be unfruitful. A $k$-strong honest definition for $\varphi$ should refine $\varphi$-types, but in (\ref{eqnwalkerSHDnew}), this is achieved not by $\psi$ alone but by $\{\psi(x;c)\}\cup\bigcup_{i=1}^k\tp(a_{\neq i}/B)$. Now, we would like our $k$-strong honest definition to be a formula rather than a type. By compactness, we know that $\{\psi(x;c)\}\cup\bigcup_{i=1}^k\tp(a_{\neq i}/B)$ can be replaced by a finite subset in (\ref{eqnwalkerSHDnew}). That is, for all finite $\bar{B}$, there are $\psi_i(x_{\neq i}; c_i)\in \tp(a_{\neq i}/B)$ such that (\ref{eqnwalkerSHDnew}) can be replaced by
\[a\models \psi(x;c)\wedge\bigwedge_{i=1}^k \psi_i(x_{\neq i}; c_i)\vdash \varphi(x;b) \leftrightarrow \varphi(a;b).\]

It appears as if we have our `non-uniform' $k$-strong honest definition $\psi\wedge \bigwedge_{i=1}^k \psi_i$, but the reader must not forget that the choice of the $\psi_i$ here depends on $\bar{B}\subseteq B$. To remove this dependence, we need to do some work. Our first goal is the following theorem.

\begin{theorem}\label{nonuniformkSHD}
    Let $k\in \N^+$. The following are equivalent.
    \begin{enumerate}[(i)]
        \item The theory $T$ is strongly $k$-distal.
        \item Let $\varphi(x_1, ..., x_k; y)\in L$, $\emptyset\neq B\subseteq M\models T$, and $a=(a_1, ..., a_k)\in M$. Write $x:=(x_1, ..., x_k)$. Then there are $\psi_i(x_{\neq i},y,z_i)\in L$ for $i\in [k]$, $\psi_{k+1}(x; z_{k+1})\in L$, and $N\in\N$, such that for all finite $\bar{B}\subseteq B$, there are $c^{(j)}_1, ..., c^{(j)}_{k+1}\in B$ for $j\in [N]$, such that for all $b\in \bar{B}$, there is $j\in [N]$ with
            \[a\models \psi_{k+1}(x,c^{(j)}_{k+1})\wedge\bigwedge_{i=1}^k \psi_i(x_{\neq i},b,c^{(j)}_i)\vdash \varphi(x;b)\leftrightarrow \varphi(a;b).\]
        \item Let $\varphi(x_1, ..., x_k; y)\in L$, $\emptyset\neq B\subseteq M\models T$, and $a=(a_1, ..., a_k)\in M$. Write $x:=(x_1, ..., x_k)$. Let $(M',B')\succcurlyeq (M,B)$ be $|M|^+$-saturated. Then there are $\psi_i(x_{\neq i},y,z_i)\in L$ for $i\in [k]$, $\psi_{k+1}(x; z_{k+1})\in L$, $N\in\N$, and $c^{(j)}_1, ..., c^{(j)}_{k+1}\in B'$ for $j\in [N]$, such that for all $b\in B$, there is $j\in [N]$ with
            \[a\models \psi_{k+1}(x,c^{(j)}_{k+1})\wedge\bigwedge_{i=1}^k \psi_i(x_{\neq i},b,c^{(j)}_i)\vdash \varphi(x;b)\leftrightarrow \varphi(a;b).\]
    \end{enumerate}
\end{theorem}
Note that, in (ii) and (iii) of Theorem \ref{nonuniformkSHD}, $(\psi_1, ..., \psi_{k+1})$ acts as a `non-uniform' $k$-strong honest definition for $(\varphi,a,B)$; recall that `non-uniformity' refers to its dependence on $a$ and $B$. After proving Theorem \ref{nonuniformkSHD}, we will bootstrap it to generate `uniform' $k$-strong honest definitions for $\varphi$ --- ones that depend only on $\varphi$ and not $a$ or $B$ --- under an extra NIP assumption. These will be defined precisely in Definition \ref{defnkSHDlegit}.

\begin{remark}
    In (ii) and (iii) of Theorem \ref{nonuniformkSHD}, the awkward parameter $N\in\N$ arises from a coding process, when we construct $\psi_1, ..., \psi_k$ each as a code for multiple formulas. We are not able to obtain a statement without such $N$. However, the reader is invited to check that the formula $\psi_{k+1}(x, c^{(j)}_{k+1})$ can be replaced by the formula $\bigwedge_{j=1}^N\psi_{k+1}(x, c^{(j)}_{k+1})$, and so the $N$ parameters $c^{(1)}_{k+1}, ..., c^{(N)}_{k+1}$ can be combined into one parameter $c_{k+1}:=(c^{(1)}_{k+1}, ..., c^{(N)}_{k+1})$; we thank Mira Tartarotti for pointing this out to us. We will nonetheless persist with our original formulation to avoid talking about $c_{k+1}$ separately from $c^{(j)}_1, ..., c^{(j)}_k$.
    
    We will make further comments on the parameter $N$ after defining `uniform' $k$-strong honest definitions (which also make reference to such $N$) in Definition \ref{defnkSHDlegit}.
\end{remark}

Towards proving Theorem \ref{nonuniformkSHD}, we appeal to the following result of Walker. For a type $q\in S(A)$ and $A_0\subseteq A$, write $q|_{A_0}:= q\cap L(A_0)$. Recall that a type $q$ is \textit{finitely satisfiable} over a set $B\subseteq U$ if, for all $\varphi(x)\in q$, there is $b\in B$ such that $\models \varphi(b)$. 

\patchcmd{\thmhead}{(#3)}{#3}{}{}
\begin{lemma}[{\cite[Lemma 9.12]{walkerthesis}}]\label{walkerSHDlemma}
    Suppose $T$ is strongly $k$-distal. Let $B\subseteq M\models T$ and $a=(a_1, ..., a_k)\in M$. Let $p:=\tp(a/\monster)$, and for all $i\in [k]$, let $p_{\neq i}:=\tp(a_{\neq i}/\monster)$. Let $(M',B')\succcurlyeq (M,B)$ be $|M|^+$-saturated. Then, for all $q\in S(\monster)$ finitely satisfiable over $B$,
    \[p|_{B'}\cup \bigcup_{i=1}^k (p_{\neq i}\otimes q)|_{B'}\vdash (p\otimes q)|_{B'}.\]
\end{lemma}
For a tuple of variables $y$ and $B\subseteq B'\subseteq \monster$, where $B$ is small but $B'$ is not necessarily small, write $S^\fs_y(B';B):=\{p(y)\in S_y(B'): p\text{ is finitely satisfiable over }B\}$. We require the following lemmas about finitely satisfiable types. 
\begin{lemma}\label{fsextension}
    Let $y$ be a tuple of variables. Let $p(y)$ be a partial type that is finitely satisfiable over a small set $B\subseteq \monster$. Then $p$ extends to a (complete) global type that is finitely satisfiable over $B$. Thus, if $B\subseteq B'$ for some not necessarily small set $B'\subseteq \monster$, then
    \[S^\fs_y(B';B)=\{q|_{B'}: q\in S^\fs_y(\monster;B)\}.\]
\end{lemma}
\begin{proof}
    We follow the argument in \cite[Section 2.2]{NIPguide}. Since $p$ is finitely satisfiable over $B$, we can extend the set $\{\varphi(B): \varphi(y)\in p\}$ to an ultrafilter $\mathcal{F}$ on $B^y$. Then $p$ extends to the global type $\{\varphi(y)\in L(\monster): \varphi(B)\in\mathcal{F}\}$, which is finitely satisfiable over $B$.
\end{proof}
\begin{lemma}\label{fslemma}
    Let $x,y$ be tuples of variables. Let $a\in \monster^x$ and $B\subseteq B'\subseteq \monster$, where $B$ is small but $B'$ is not necessarily small. Let $p(x)=\tp(a/\monster)$. Then the following hold.
    \begin{enumerate}[(i)]
        \item If $q\in S_y(\monster)$ is $B$-invariant, then $p\otimes q=q\otimes p =\{\varphi(x,y):\varphi(a,y)\in q\}$.
        \item The set $S^\fs_{a,y}(B';B):=\{(p\otimes q)|_{B'}: q\in S^\fs_y(\monster; B)\}$ is closed in $S_{x,y}(B')$.
    \end{enumerate}
\end{lemma}
\begin{proof}
    (i) It suffices to show that $p\otimes q=\{\varphi(x,y):\varphi(a,y)\in q\}$. Let $\varphi(x,y)\in L(C)$, where $\{a\}\cup B\subseteq C$, and let $b\models q|_C$. Then
    \[\varphi(x,y)\in p\otimes q \Leftrightarrow \varphi(x,b)\in p\Leftrightarrow \monster\models \varphi(a,b)\Leftrightarrow \varphi(a,y)\in q.\]
    
    (ii) Without loss of generality, suppose $B\neq\emptyset$. It suffices to show that for $r\in S_{x,y}(B')$, $r\in S^\fs_{a,y}(B';B)$ if and only if whenever $\varphi(x,y)\in L(B')$ is such that $\{\varphi(a,y)\}$ is not finitely satisfiable over $B$, then $\neg \varphi(x,y)\in r$.

    Suppose $r\in S^\fs_{a,y}(B';B)$, so we have that $r=(p\otimes q)|_{B'}$ for some $q\in S^\fs_y(\monster; B)$. If $\varphi(x,y)\in L(B')$ is such that $\varphi(x,y)\in r$, then $\varphi(a,y)\in q$ and so $\{\varphi(a,y)\}$ is finitely satisfiable over $B$. Conversely, suppose whenever $\varphi(x,y)\in L(B')$ is such that $\{\varphi(a,y)\}$ is not finitely satisfiable over $B$, then $\neg \varphi(x,y)\in r$. Then $r(a,y)$ is finitely satisfiable over $B$, so extends to some $q\in S^\fs_y(\monster; B)$ by Lemma \ref{fsextension}. But then $r=(p\otimes q)|_{B'}$: these are complete types such that if $\varphi(x,y)\in r$, then $\varphi(a,y)\in r(a,y)\subseteq q$, and so $\varphi(x,y)\in p\otimes q$.
\end{proof}
\patchcmd{\thmhead}{#3}{(#3)}{}{}
We are now ready to prove Theorem \ref{nonuniformkSHD}.
\begin{proof}[Proof of Theorem \ref{nonuniformkSHD}]
    Firstly, we argue that (i) implies (iii). Suppose $T$ is strongly $k$-distal. Let $\varphi(x_1, ..., x_k; y)\in L$, $B\subseteq M\models T$, and $a=(a_1, ..., a_k)\in M$. If $|B|\leq 1$ then (iii) holds trivially, so assume $|B|\geq 2$. Write $x:=(x_1, ..., x_k)$. Let $(M',B')\succcurlyeq (M,B)$ be $|M|^+$-saturated. Let $p:=\tp(a/\monster)$, and for all $i\in [k]$, let $p_{\neq i}:=\tp(a_{\neq i}/\monster)$. 

    Let $q\in S^\fs_y(\monster;B)$. By Lemma \ref{walkerSHDlemma}, there is $\varepsilon_q\in\{0,1\}$ such that
    \[r^q:=p|_{B'}\cup\bigcup_{i=1}^k (p_{\neq i}\otimes q)|_{B'}\vdash \varphi^{\varepsilon_q}(x;y).\]
    By compactness, there are $\psi^q_{k+1}(x,c^q_{k+1})\in p|_{B'}$ and $\psi^q_i(x_{\neq i},y,c^q_i)\in (p_{\neq i}\otimes q)|_{B'}$ for $i\in [k]$ such that
    \[\psi^q:=\psi^q_{k+1}(x,c^q_{k+1})\wedge\bigwedge_{i=1}^k \psi^q_i(x_{\neq i},y,c^q_i)\vdash \varphi^{\varepsilon_q}(x;y).\]
    
    Now, $\{[\psi^q]: q\in S^\fs_y(\monster;B)\}$ is an open cover for $S^\fs_{a,y}(B';B)$. By Lemma \ref{fslemma}, $S^\fs_{a,y}(B';B)$ is a closed, hence compact, subset of $S_{x,y}(B')$, so the open cover above has a finite subcover $\{[\psi^q]: q\in Q\}$.

    For all $b\in B$, we have $\tp(b/\monster)\in S^\fs_y(\monster;B)$, so there is $q(b)\in Q$ such that $\psi^{q(b)}\in p\otimes \tp(b/\monster)=\tp(a,b/\monster)$, whence
    \[a\models \psi^{q(b)}_{k+1}(x,c^{q(b)}_{k+1})\wedge\bigwedge_{i=1}^k \psi^{q(b)}_i(x_{\neq i},b,c^{q(b)}_i)\vdash \varphi^{\varepsilon_{q(b)}}(x; b);\]
    in particular, $\models \varphi^{\varepsilon_{q(b)}}(a; b)$, and so
    \[a\models \psi^{q(b)}_{k+1}(x,c^{q(b)}_{k+1})\wedge\bigwedge_{i=1}^k \psi^{q(b)}_i(x_{\neq i},b,c^{q(b)}_i)\vdash \varphi(x; b)\leftrightarrow \varphi(a;b).\]
    For all $i\in [k+1]$, we can code $(\psi^q_i: q\in Q)$ into a single formula as follows: for all $b \in B$,
    \begin{align*}
        a\models \bigvee_{q\in Q}\left(\psi^q_{k+1}(x,c^q_{k+1})\wedge u_{k+1}^q=v_{k+1}^q\right)\wedge \bigwedge_{i=1}^k \bigvee_{q\in Q}&\left(\psi^q_i(x_{\neq i},b,c^q_i)\wedge u_i^q=v_i^q\right)\\
        &\;\;\;\;\;\;\;\;\;\;\;\;\;\;\;\vdash \varphi(x; b)\leftrightarrow \varphi(a;b),
    \end{align*}
    for any $u^q_1, ..., u^q_{k+1}, v^q_1, ..., v^q_{k+1}\in B'$ such that for all $i\in [k+1]$, $u^q_i=v^q_i$ if and only if $q=q(b)$; such $u^q, v^q$ exist since $|B|\geq 2$. Therefore, (iii) holds.

    Next, we argue that (iii) implies (ii). Our argument expands that in \cite[Corollary 9]{extdefinable}. Suppose (iii) holds. Let $\varphi(x_1, ..., x_k;y)\in L$, $\emptyset\neq B\subseteq M\models T$, and $a\in M$. Let $(M',B')\succcurlyeq (M,B)$ be any $|M|^+$-saturated elementary extension, and let $\psi_1, ..., \psi_{k+1}$ and $N$ be given by (iii). Then, for all finite $\bar{B}\subseteq B$, $(M', B')$ satisfies the first-order formula saying that there are $c^{(j)}_1, ..., c^{(j)}_{k+1}\in B'$ for $j\in [N]$ satisfying the conclusion of (ii). Since $(M',B')\succcurlyeq (M,B)$ is an elementary extension, $(M, B)$ satisfies the same formula with $B'$ replaced by $B$ throughout, so (ii) holds.

    Finally, we argue that (ii) implies (i). By Theorem \ref{walkerSHDnew}, it suffices to show that (ii) implies statement (ii) of Theorem \ref{walkerSHDnew}. Let $\varphi(x_1, ..., x_k; y)\in L$, $\emptyset\neq B\subseteq M\models T$, and $a=(a_1, ..., a_k)\in M$. Let $\psi_{k+1}(x;z_{k+1})$ and $N$ be given by (ii), and let $\psi(x; z_{k+1}^{(1)}, ..., z_{k+1}^{(N)}):=\bigwedge_{j=1}^N \psi_{k+1}(x;z_{k+1}^{(j)})$. Then, for all finite $\bar{B}\subseteq B$, there are $c_{k+1}^{(j)}\in B$ for $j\in [N]$ such that, for all $b\in\bar{B}$,
    \[a\models \{\psi(x;c_{k+1}^{(1)}, ..., c_{k+1}^{(N)})\}\cup\bigcup_{i=1}^k\tp(a_{\neq i}/B)\vdash \varphi(x;b)\leftrightarrow \varphi(a;b)\]
    as required.
\end{proof}

Our next goal is to bootstrap Theorem \ref{nonuniformkSHD} to generate `uniform' $k$-strong honest definitions. It is now clear what these should look like.

\begin{defn}\label{defnkSHDlegit}
    Let $\varphi(x_1, ..., x_k;y)\in L$; write $x:=(x_1, ..., x_k)$. Let $N\in\N$. A $(k+1)$-tuple of $L$-formulas $(\psi_i(x_{\neq i},y,z_i):i\in [k])^\frown (\psi_{k+1}(x,z_{k+1}))$ is a \textit{$k$-strong honest definition} for $\varphi$ of \textit{degree} $N$ if the following holds.

    Let $B\subseteq M\models T$ with $2\leq |B|<\infty$ and $a=(a_1, ..., a_k)\in M$. Then there are $c^{(j)}_1, ..., c^{(j)}_{k+1}\in B$ for $j\in [N]$ such that for all $b\in B$, there is $j\in [N]$ with
            \[a\models \psi_{k+1}(x,c^{(j)}_{k+1})\wedge\bigwedge_{i=1}^k \psi_i(x_{\neq i},b,c^{(j)}_i)\vdash \varphi(x;b)\leftrightarrow \varphi(a;b).\]
\end{defn}
\begin{remark}\label{remarkkSHD}
    By compactness, $(\psi_i(x_{\neq i},y,z_i):i\in [k])^\frown(\psi_{k+1}(x,z_{k+1}))$ is a \textit{$k$-strong honest definition} for $\varphi$ of \textit{degree} $N$ if and only if the following holds.

    Let $B\subseteq M\models T$ with $|B|\geq 2$ and $a\in M$. Let $(M',B')\succcurlyeq (M,B)$ be $|M|^+$-saturated. Then there are $c^{(j)}_1, ..., c^{(j)}_{k+1}\in B'$ for $j\in [N]$ such that for all $b\in B$, there is $j\in [N]$ with
            \[a\models \psi_{k+1}(x,c^{(j)}_{k+1})\wedge\bigwedge_{i=1}^k \psi_i(x_{\neq i},b,c^{(j)}_i)\vdash \varphi(x;b)\leftrightarrow \varphi(a;b).\]
\end{remark}
Note that a 1-strong honest definition of degree 1 is a strong honest definition, and if $\psi(x,z)$ is a 1-strong honest definition of degree $N>1$, then $\bigwedge_{j=1}^N \psi(x,z_j)$ is a (1-)strong honest definition (of degree 1). Hence, when defining strong honest definitions, we did not need to make reference to degrees. Sadly, when $k\geq 2$, this trick does not work for $k$-strong honest definitions, and we do not know whether the reference to degrees can be eliminated.
\begin{problem}\label{qndegrees}
    If $\varphi(x_1, ..., x_k; y)\in L$ has a $k$-strong honest definition, does it have a $k$-strong honest definition of degree 1?
\end{problem}
As we shall see, the reference to degrees does not seem to affect the efficacy of $k$-strong honest definitions, it merely makes the proofs more awkward.

The main result of this section is as follows.
\begin{theorem}\label{uniformkSHD}
    Let $T$ be NIP and let $k\in \N^+$. The following are equivalent.
    \begin{enumerate}[(i)]
        \item The theory $T$ is strongly $k$-distal.
        \item Every $\varphi(x_1, ..., x_k;y)\in L$ has a $k$-strong honest definition.
    \end{enumerate}
\end{theorem}
Our proof bootstraps the `non-uniform' Theorem \ref{nonuniformkSHD}, following the strategy in \cite[Theorem 21]{extdefinable}. The ingredient that necessitates NIP is the following fact.
\patchcmd{\thmhead}{(#3)}{#3}{}{}
\begin{fact}[{\cite[Theorem 4; $(p,q)$-theorem]{pqthm}}]\label{pqthm}
    For all $p\geq q\in\N^+$, there is $K=K(p,q)\in\N^+$ such that the following holds.
    
    Let $\mathcal{F}\subseteq \mathcal{P}(X)$ be a finite family with $\VC^*(\mathcal{F})\leq q$, and suppose $\mathcal{F}$ has the \textit{$(p,q)$-property}: if $\mathcal{F}_0\subseteq \mathcal{F}$ has size $p$, then there is a $q$-element subset of $\mathcal{F}_0$ with non-empty intersection. Then there is $Y\subseteq X$ of size at most $K(p,q)$ such that $F\cap Y\neq\emptyset$ for all $F\in\mathcal{F}$.
\end{fact}
\patchcmd{\thmhead}{#3}{(#3)}{}{}
That $K$ only depends on $p$ and $q$ in Fact \ref{pqthm} is not stated explicitly in \cite[Theorem 4]{pqthm}; see \cite[Remark 7]{extdefinable} for an argument to this end.

We prove a simple compactness lemma.
\begin{lemma}\label{uniformlemma}
    Let $T$ be strongly $k$-distal and $\varphi(x_1, ..., x_k; y)\in L$. Write $x:=(x_1, ..., x_k)$. For all $(k+1)$-tuples of $L$-formulas $\Psi:=(\psi_i(x_{\neq i},y,z_i):i\in [k])^\frown (\psi_{k+1}(x,z_{k+1}))$ and $N\in\N$, fix $m_{\Psi,N}\in\N$. Then there are $H, N_1, ..., N_H\in\N$ and $(k+1)$-tuples of $L$-formulas $\Psi^{(h)}:=(\psi^{(h)}_i(x_{\neq i},y,z^{(h)}_i):i\in [k])^\frown (\psi^{(h)}_{k+1}(x,z^{(h)}_{k+1}))$ for $h\in [H]$ such that the following holds.
    
    Let $B\subseteq M\models T$ with $|B|\geq 2$, and let $a\in M$. Then there is $h\in [H]$ such that, for all $\bar{B}\subseteq B$ of size at most $m_{\Psi^{(h)},N_h}$, there are $c^{(j)}_1, ..., c^{(j)}_{k+1}\in B$ for $j\in [N_h]$ such that for all $b\in \bar{B}$, there is $j\in [N_h]$ with
    \[a\models \psi^{(h)}_{k+1}(x,c^{(j)}_{k+1})\wedge \bigwedge_{i=1}^k \psi^{(h)}_i(x_{\neq i},b,c^{(j)}_i)\vdash \varphi(x;b)\leftrightarrow \varphi(a;b).\]
\end{lemma}
\begin{proof}
    Let $P$ be a new unary predicate. Let $T'$ be the theory in the language $L':=L\cup \{P,a\}$ saying that if $(M,B,a)\models T'$, then $M\models T$, $|B|\geq 2$, and for every $(k+1)$-tuple of $L$-formulas $\Psi:=(\psi_i(x_{\neq i},y,z_i):i\in [k])^\frown (\psi_{k+1}(x,z_{k+1}))$ and $N\in\N$, there is $\bar{B}\subseteq B$ of size at most $m_{\Psi,N}$, for which there are no $c^{(j)}_1, ..., c^{(j)}_{k+1}\in B$ for $j\in [N]$ such that for all $b\in\bar{B}$, there is $j\in [N]$ with
     \[a\models \psi_{k+1}(x,c^{(j)}_{k+1})\wedge \bigwedge_{i=1}^k \psi_i(x_{\neq i},b,c^{(j)}_i)\vdash \varphi(x;b)\leftrightarrow \varphi(a;b).\]
    By Theorem \ref{nonuniformkSHD}, $T'$ is inconsistent. By compactness, $T'$ is finitely inconsistent, giving the existence of $N_1, ..., N_H$ and $\Psi^{(1)}, ..., \Psi^{(H)}$ as claimed.
\end{proof}
We now prove Theorem \ref{uniformkSHD}.
\begin{proof}[Proof of Theorem \ref{uniformkSHD}]
    That (ii) implies (i) follows immediately from Theorem \ref{nonuniformkSHD}, so we prove that (i) implies (ii).
    
    Suppose $T$ is strongly $k$-distal. Let $\varphi(x_1, ..., x_k;y)\in L$, $x:=(x_1, ..., x_k)$, and $d:=|y|$. For each $(k+1)$-tuple of $L$-formulas $\Psi:=(\psi_i(x_{\neq i},y,z_i):i\in [k])^\frown (\psi_{k+1}(x,z_{k+1}))$  and $N\in\N$, let $\theta_{\Psi,N}(z^{(1)}, ..., z^{(N)};x,y)$ be the following formula, where for all $j\in [N]$ we have $z^{(j)}=(z^{(j)}_1, ..., z^{(j)}_{k+1})$:
    \begin{align*}
        \bigvee_{j=1}^N\left(\psi_{k+1}(x,z^{(j)}_{k+1})\wedge\bigwedge_{i=1}^k  \psi_i(x_{\neq i},y,z^{(j)}_i)\phantom{\left(\left(\bigvee_1^k\right)\right)}\right.\;\;\;\;\;\;\;\;\;\;\;\;\;\;\;\;\;\;\;\;\;\;\;\;\;\;\;\;\;\;\;\;\;\;\;\;\;\;\;\;\;\;\;\;\;\;\\
    \left.\phantom{a}\wedge\forall x'\left(\left(\psi_{k+1}(x',z^{(j)}_{k+1})\wedge\bigwedge_{i=1}^k  \psi_i(x'_{\neq i},y,z^{(j)}_i)\right)\to(\varphi(x';y)\leftrightarrow \varphi(x;y))\right)\right);
    \end{align*}
    let $m_{\Psi,N}:=d\cdot \VC^*(\theta_{\Psi,N})\in\N$. That is, $\theta_{\Psi,N}(c^{(1)}, ..., c^{(N)};a,b)$ holds if and only if there is $j\in [N]$ such that
    \[a\models \psi_{k+1}(x,c^{(j)}_{k+1})\wedge \bigwedge_{i=1}^k \psi_i(x_{\neq i},b,c^{(j)}_i)\vdash \varphi(x;b)\leftrightarrow \varphi(a;b).\]
    
    By standard coding tricks, we may apply Lemma \ref{uniformlemma} under the assumption that $H=1$. (Otherwise, the following proof produces $(\psi^{(h)}_1, ..., \psi^{(h)}_{k+1})_{h\in [H]}$ such that, for all $a$ and $B$, there is $h\in [H]$ such that $(\psi^{(h)}_1, ..., \psi^{(h)}_{k+1})$ works; then, as in the proof of Theorem \ref{nonuniformkSHD}, for each $i\in [k+1]$ we code the formulas $(\psi^{(h)}_i: h\in [H])$ into a single formula $\psi_i$ such that, for all $a$ and $B$, $(\psi_1, ..., \psi_{k+1})$ works.)

    Applying Lemma \ref{uniformlemma} with $H=1$, we obtain the $(k+1)$-tuple $\Psi^{(1)}:=(\psi^{(1)}_i(x_{\neq i},y,z^{(1)}_i):i\in [k])^\frown (\psi^{(1)}_{k+1}(x,z^{(1)}_{k+1}))$ --- from which we shall henceforth drop the superscripts --- and $N:=N_1\in\N$. Let $e:=|z_1|+\cdots+|z_{k+1}|$. Let $B\subseteq M\models T$ with $2\leq |B|<\infty$, and let $a\in M^x$. For $b\in B^y$, $\theta_{\Psi,N}(B^{eN};a,b)$ is the set
    \begin{align*}
        &\left\lbrace (c^{(1)}, ..., c^{(N)})\in B^{eN}:\phantom{\bigwedge_{i=1}^k}\right.\\
        &\;\;\;\;\;\left.a\models \psi_{k+1}(x,c^{(j)}_{k+1})\wedge \bigwedge_{i=1}^k \psi_i(x_{\neq i},b,c^{(j)}_i)\vdash \varphi(x;b)\leftrightarrow \varphi(a;b)\text{ for some }j\in [N]\right\rbrace.
    \end{align*}
    
    Observe that the family $\mathcal{F}:=\{\theta_{\Psi,N}(B^{eN};a,b): b\in B^y\}\subseteq \mathcal{P}(B^{eN})$ has the $(m_{\Psi,N}/d, m_{\Psi,N}/d)$-property, that is, any subset of $\mathcal{F}$ of size $m_{\Psi,N}/d$ has non-empty intersection. Indeed, given $b_1, ..., b_{m_{\Psi,N}/d}\in B^y$, there is $\bar{B}\subseteq B$ of size at most $m_{\Psi,N}$ such that $b_1, ..., b_{m_{\Psi,N}/d}\in \bar{B}^y$, and our choice of $\Psi$ (given by Lemma \ref{uniformlemma}) is precisely such that there is $(c^{(1)}, ..., c^{(N)})\in \bigcap_{i\in [m_{\Psi,N}/d]} \theta_{\Psi,N}(B^{eN};a,b_i)$.
    
    By standard VC-(co)dimension manipulations, $\VC^*(\mathcal{F})\leq \VC^*(\theta_{\Psi,N})=m_{\Psi,N}/d$. By the $(p,q$)-theorem (Fact \ref{pqthm}), there is $Y\subseteq B^{eN}$ of size at most $K=K(m_{\Psi,N}/d, m_{\Psi,N}/d)\in\N$, such that $F\cap Y\neq \emptyset$ for all $F\in\mathcal{F}$. That is, there are $c^{(1)}, ..., c^{(KN)}\in B^e$ such that for all $b\in B^y$, there is $j\in [KN]$ with
    \[a\models \psi_{k+1}(x,c^{(j)}_{k+1})\wedge \bigwedge_{i=1}^k \psi_i(x_{\neq i},b,c^{(j)}_i)\vdash \varphi(x;b)\leftrightarrow \varphi(a;b).\]

    Since the above holds for all $B\subseteq M\models T$ with $2\leq |B|<\infty$ and $a\in M$, we conclude that $(\psi_1, ..., \psi_{k+1})$ is a $k$-strong honest definition for $\varphi$ (of degree $KN$).
\end{proof}
We have shown that, under a global NIP assumption, the existence of $k$-strong honest definitions characterises strong $k$-distality. Since there are strongly $k$-distal theories that are not NIP, it is natural to pose the following problem.
\begin{problem}\label{qnnoNIP1}
    Can the NIP assumption be removed from Theorem \ref{uniformkSHD}?
\end{problem}

Although strongly $k$-distal theories need not be NIP, they are $\text{NIP}_k$ --- a proof of this is given in {\cite[Proposition 6.7]{walker}}, attributed to Chernikov. Thus, one may hope that all uses of NIP can be replaced with uses of $\text{NIP}_k$. However, among other things, this would require an $\text{NIP}_k$ version of the $(p,q)$-theorem, which is yet to be developed; even the statement of such a theorem is not obvious.

The regularity lemma we shall derive in the next section is for all hypergraphs defined by a formula $\varphi(x_1, ..., x_k; y)$ with a $k$-strong honest definition in an NIP theory. In particular, we do not require the full strength of strong $k$-distality, since we only require the formula in question to have a $k$-strong honest definition, rather than all formulas. However, we still require the theory to be NIP, mainly so that we may invoke an $\varepsilon$-approximation result (Theorem \ref{epsilonapprox}); see the discussion after Problem \ref{qnnoNIP2} for a summary of the use of NIP.
\section{Regularity lemma}\label{seckdistalreglemma}
In this section, we state and prove our regularity lemma for $(k+1)$-ary formulas with a $k$-strong honest definition in an NIP theory. We begin with some background work.
\subsection{Keisler measures}
Throughout this subsection, fix a complete $\ELL$-theory $T$ and a sufficiently saturated model $\monster\models T$.

Recall that the reglarity lemmas we have seen so far for finite hypergraphs take the following form. Given a finite hypergraph $(\bigsqcup_{i=1}^k V, E)$, $V^k$ can be partitioned into boundedly many parts, such that the non-quasirandom parts make up a small fraction of $V^k$. In the model-theoretic setting, we would take $V$ to be a subset of $\monster^d$ for some $d\in\N^+$. If $V$ is finite and we put a normalised finite counting measure on $\monster^d$ whose support is $V$, then the regularity lemma would say that the non-quasirandom parts have small measure.

There is a much more natural and general class of measures that one can put on $\monster^d$: \textit{(global) Keisler measures}. To define these, we will write, for $x$ a tuple of variables and $B\subseteq\monster$, $\mathcal{L}_x(B)$ as the Boolean algebra of $B$-definable subsets of $\monster^x$, often representing an element of $\mathcal{L}_x(B)$ by an $L(B)$-formula defining it.
\begin{defn}\label{defnkeisler}
    Let $x$ be a tuple of variables and $B\subseteq \monster$. A function $\mu(x): \mathcal{L}_x(B)\to [0,1]$ is a \textit{Keisler measure over $B$} if it is a finitely additive probability measure, that is:
    \begin{enumerate}[(i)]
        \item $\mu(x=x)=1$;
        \item for all $\varphi(x)\in \mathcal{L}_x(B)$, $\mu(\neg\varphi(x))=1-\mu(\varphi(x))$;
        \item for all disjoint $\varphi_1(x) ,..., \varphi_k(x)\in \mathcal{L}_x(B)$, $\mu\left(\bigvee_{i=1}^k \varphi_i(x)\right)=\sum_{i=1}^k \mu(\varphi_i(x))$.
    \end{enumerate}
    If $B=\monster$, say that $\mu$ is a \textit{global Keisler measure}.
    
    For small $A\subseteq \monster$, say that a global Keisler measure $\mu$ is \textit{invariant over $A$} or \textit{$A$-invariant} if, for all $\varphi(x;y)\in L$ and $d,d'\in\monster$ such that $d\equiv_A d'$, $\mu(\varphi(x;d))=\mu(\varphi(x;d'))$.
\end{defn}

Given $B\subseteq\monster$ and a Keisler measure $\mu(x): \mathcal{L}_x(B)\to [0,1]$, there is a unique sensible way to extend $\mu$ to a $\sigma$-additive probability measure on the Stone space $S_x(B)$ of types in $x$ over $B$; see \cite[Section 7.1]{NIPguideonline} for the details. Abusing notation, we will conflate $\mu(x)$ with this extension.
    
We will not give much exposition of Keisler measures here, other than collecting a few facts necessary for this section; we refer the reader to \cite[Chapter 7]{NIPguide} for more details. A health warning for the previous reference is appropriate here: the exposition is in the context of an NIP theory (although the definition of Keisler measures and some of the results apply in general). Indeed, Keisler measures have much nicer properties in an NIP theory. For instance, to define products of Keisler measures as follows, we need to work in an NIP theory.
\begin{defn}\label{defnproductmeasure}
    Suppose $T$ is NIP. Let $\mu(x),\lambda(y)$ be global Keisler measures, with $\mu$ invariant over some small $M\models T$. The \textit{product} $(\mu\otimes \lambda)(x,y)$ is the global Keisler measure such that, for all $\varphi(x,y;b)\in L(\monster)$,
    \begin{equation}\label{eqnmeasureproduct}
        (\mu\otimes\lambda)(\varphi(x,y;b))=\int_{S_y(N)}f\mathop{}\! \mathrm{d}\lambda|_N,
    \end{equation}
    where $N$ is any small model containing $M\cup \{b\}$ and $f: S_y(N)\to [0,1]$ sends $q\in S_y(N)$ to $\mu(\varphi(x,d;b))$ for any $d\models q$.
\end{defn}

For (\ref{eqnmeasureproduct}) to be well-defined, it must be independent of the choice of $N$, and the function $f$ must be measurable. For the latter to be true, it is crucial that $T$ is NIP; see \cite[Section 7.4]{NIPguide}. We will use the following facts about the product operation, also found in \cite[Section 7.4]{NIPguide}.
\begin{fact}
    Suppose $T$ is NIP. Let $\mu(x),\lambda(y), \nu(z)$ be global Keisler measures, with $\mu, \lambda$ invariant over some small $M\models T$.
    \begin{enumerate}[(i)]
        \item The product $\mu\otimes \lambda$ is invariant over $M$.
        \item The product operation is associative: $\mu\otimes (\lambda\otimes \nu)=(\mu\otimes \lambda)\otimes \nu$.
    \end{enumerate}
    In light of (ii), we define $\mu^{(n)}(x_1, ..., x_n):=\mu(x_1)\otimes\dots\otimes \mu(x_n)$ for $n\in\N^+$.
\end{fact}

We will be interested in a particularly well-behaved subclass of global Keisler measures: those that are \textit{generically stable} over a small model $M\models T$. Such measures are invariant over $M$, and in an NIP theory, satisfy an extremely useful `$\varepsilon$-approximation' equivalence result, Theorem \ref{epsilonapprox} below. For the purposes of this paper, this result can be taken as the definition of generically stable measures; the reader is referred to \cite[Section 7.5]{NIPguide} for more exposition (in the NIP context). Roughly speaking, the $\varepsilon$-approximation result says that given a generically stable measure $\mu$ and a uniformly definable family $\mathcal{F}$ of sets, there are points $a_1, ..., a_n$ such that, for each $F\in\mathcal{F}$, $\mu(F)$ can be approximated by `sampling' $F$ on $\{a_1, ..., a_n\}$. Moreover, the number of points $n$ is bounded in terms of $\VC(\mathcal{F})$ and the error parameter. We use the following notation to capture this notion of sampling.
\begin{defn}
    Let $d\in\N$, $X$ be a set, and $F\subseteq X$. For $n\in\N^+$ and $a_1, ..., a_n\in X$, let
    \[\Av(\{a_1, ..., a_n\}; F):=\Av(a_1, ..., a_n; F):=\frac{\#{\{i\in [n]: a_i\in F\}}}{n}.\]
    Note that, in the above, $\{a_1, ..., a_n\}$ is treated as a multiset.
\end{defn}
We now state the $\varepsilon$-approximation result.
\begin{theorem}\label{epsilonapprox}
    Suppose $T$ is NIP, and let $M\models T$ be small. The following are equivalent for a global $M$-invariant Keisler measure $\mu(x)$.
    \begin{enumerate}[(i)]
        \item The measure $\mu$ is generically stable over $M$.
        \item Let $\varphi(x;y)\in L$ and $\varepsilon\in (0,1]$. Then there are $a_1, ..., a_n\in M^x$ such that, for all $b\in \monster$,
        \[\abs{\Av(a_1, ..., a_n; \varphi(x;b))-\mu(\varphi(x;b))}<\varepsilon.\]
    \end{enumerate}
    Moreover, in (ii), we may assume $n=O_{\VC(\varphi)}(\varepsilon^{-2}\log (2\varepsilon^{-1}))$; in particular, $n$ can be chosen independently of $\mu$ and $M$.
\end{theorem}
\begin{proof}
    See \cite[Theorem 7.29]{NIPguide}. The `moreover' statement is a combination of \cite[Lemma 7.24]{NIPguide} and the Sauer--Shelah Lemma (see, for example, \cite[Lemma 6.4]{NIPguide}).
\end{proof}
The finite combinatorialists among us will be relieved to know that, in an NIP theory, every normalised finite counting measure is generically stable (over any small model containing its support) --- see (the proof of) \cite[Corollary 6.9]{NIPguide} --- so the regularity lemmas that we prove can be applied to finite hypergraphs. Other notable examples of generically stable measures in the NIP context include pseudofinite counting measures, averages of types, and average measures of indiscernible sequences --- see \cite[Example 7.32]{NIPguide}.

In preparation for what follows, we prove that generic stability is preserved under relativisation. Given a global Keisler measure $\mu(x)$ and a definable $V\subseteq \monster^x$ with $\mu(V)>0$, we let the \textit{relativisation of $\mu$ to $V$}, denoted $\mu|_V(x)$, be such that $\mu|_V(\varphi(x)):=\mu(\varphi(x)\cap V)/\mu(V)$ for all $\varphi(x)\in L(\monster)$. It is an easy exercise to check that this defines a global Keisler measure.
\begin{prop}\label{relativisation}
    Suppose $T$ is NIP, and let $M\models T$ be small. Let $\mu(x)$ be a global Keisler measure, generically stable over $M$. Then, for all $\monster$-definable $V\subseteq \monster^x$ with $\mu(V)>0$, $\mu|_V(x)$ is generically stable over $M$.
\end{prop}
\begin{proof}
    Let $\varphi(x;y)\in L$ and $\varepsilon\in (0,1]$. Let $V=\theta(x;c)$ for $\theta(x;z)\in L$ and $c\in U$. Applying Theorem \ref{epsilonapprox} to the formula $\psi(x;y,z, w_1, w_2):=(\varphi(x;y)\vee w_1=w_2)\wedge \theta(x;z)$, there are \linebreak $a_1, ..., a_n\in M^x$ such that, for all $b\in \monster$,
        \begin{align}
            \abs{\Av(a_1, ..., a_n; \varphi(x;b)\cap V)-\mu(\varphi(x;b)\cap V)}&<\frac{\varepsilon}{2}\mu(V),\label{eqnrelativisation1}\\
            \abs{\Av(a_1, ..., a_n; V)-\mu(V)}&<\frac{\varepsilon}{2}\mu(V).\label{eqnrelativisation2}
        \end{align}
    Permuting $a_1, ..., a_n$ if necessary, we may let $m\in [n]\cup\{0\}$ be such that $a_1, ..., a_m\in V$ and $a_{m+1}, ..., a_n\not\in V$; in particular, $\Av(a_1, ..., a_n; V)=m/n$. We claim that, for all $b\in \monster$,
    \begin{equation}\label{eqnrelativisationfinal}
        \abs{\Av(a_1, ..., a_m; \varphi(x;b))-\mu|_V(\varphi(x;b))}<\varepsilon,
    \end{equation}
    which would complete the proof. Indeed, observe that for all $b\in\monster$,
    \begin{align}
        \Av(a_1, ..., a_m; \varphi(x;b))&=\Av(a_1, ..., a_m; \varphi(x;b)\cap V)\notag\\
        &=\frac{\#\{i\in [m]: a_i\in \varphi(x;b)\cap V\}}{m}\notag\\
        &=\frac{\#\{i\in [n]: a_i\in \varphi(x;b)\cap V\}}{m}\notag\\
        &=\frac{n}{m}\Av(a_1, ..., a_n; \varphi(x;b)\cap V).\label{eqnrelativisation3}
    \end{align}
    By (\ref{eqnrelativisation2}), $|\mu(V)-\frac{m}{n}|<\frac{\varepsilon}{2}\mu(V)$. Multiplying both sides by $\Av(a_1, ..., a_m; \varphi(x;b))$, we have by (\ref{eqnrelativisation3}) that for all $b\in \monster$,
    \begin{align*}
        \abs{\mu(V)\Av(a_1, ..., a_m; \varphi(x;b))-\Av(a_1, ..., a_n; \varphi(x;b)\cap V)}&<\frac{\varepsilon}{2}\mu(V)\Av(a_1, ..., a_m; \varphi(x;b))\\
        &<\frac{\varepsilon}{2}\mu(V).
    \end{align*}
    Combining with (\ref{eqnrelativisation1}),
    \[\abs{\mu(V)\Av(a_1, ..., a_m; \varphi(x;b))-\mu(\varphi(x;b)\cap V)}<\varepsilon\mu(V).\]
    Dividing both sides by $\mu(V)$, we obtain (\ref{eqnrelativisationfinal}) as required.
\end{proof}
\subsection{Main proof}
We now finally come to our first regularity lemma. This result will be bootstrapped to prove more general statements in the next subsection, but the bulk of the argument is made in this subsection.

Throughout this subsection, we fix the following.
\begin{itemize}
    \item An NIP $L$-theory $T$ and models $M,\monster\models T$ with $\monster$ sufficiently saturated and $M$ small.
    \item A formula $\varphi(x_1, ..., x_k; x_{k+1})\in L$; write $x:=(x_1, ..., x_{k+1})$.
    \item A $k$-strong honest definition $(\psi_i(x_{\neq i},z_i): i\in [k+1])$ for $\varphi$ of degree $N$.
\end{itemize}
\begin{remark}
    We had previously indexed the variables in $\varphi$ as $x_1, ..., x_k, y$, which emphasises the different roles of the $x$- and $y$-variables in the $k$-strong honest definition. In this section, our main result is a regularity lemma for the $(k+1)$-uniform hypergraph $\varphi(x_1, ..., x_{k+1})$, where, a priori, none of the variables $x_1, ..., x_{k+1}$ are special. Thus, it is sensible to index the variables as $x_1, ..., x_{k+1}$ (note, however, that $x_{k+1}$ still plays a special role in the proof). This has the added bonus of cleaner presentation. In particular, writing $x:=(x_1, ..., x_{k+1})$, a $k$-strong honest definition for $\varphi$ has the form $(\psi_i(x_{\neq i},z_i): i\in [k+1])$.
\end{remark}

Our goal is the following theorem.
\begin{theorem}\label{kSHDreglemma}
    ($T$ is NIP.) For all $\delta \in (0,1]$, there are $\theta_i(x_{\neq i},z_i)\in L$ for $i\in [k+1]$ and a natural number $K\leq \text{poly}_{\varphi,\psi_1, ..., \psi_{k+1},N}(\delta^{-1})$ such that the following holds.

    Let $\mu(x_{\neq k+1})$ and $\nu(x_{k+1})$ be Keisler measures, with $\nu(x_{k+1})$ generically stable over $M$, and let $\omega(x):=\nu(x_{k+1})\otimes\mu(x_{\neq k+1})$. Then there are partitions $\mathcal{P}_i$ of $\monster^{x_{\neq i}}$ for $i\in [k+1]$, each of size at most $K$, such that:
    \begin{enumerate}[(i)]
        \item for all $i\in [k+1]$ and $P_i\in\mathcal{P}_i$, there is $c_i\in M^{z_i}$ such that $P_i=\theta_i(x_{\neq i},c_i)$;
        \item $\sum \omega(P_1\wedge \cdots\wedge P_{k+1})\geq 1-\delta$, where the sum ranges over all $(P_1, ..., P_{k+1})\in \mathcal{P}_1\times\cdots\times \mathcal{P}_{k+1}$ such that $P_1\wedge \cdots \wedge P_{k+1}$ is $\varphi$-homogeneous.
    \end{enumerate}
\end{theorem}

In (ii), the notation of $P_1\wedge\cdots\wedge P_{k+1}$ can be understood by conflating $P_i$ with the formula that defines it, given by (i). That is,
\[P_1\wedge\cdots\wedge P_{k+1}=\{x\in \monster: x_{\neq i}\in P_i\text{ for all }i\in [k+1]\}.\]

Our proof strategy follows that of \cite[Theorem 5.8]{regularitylemma}, but it is more efficient --- we will discuss this after the proof.
\begin{defn}
    Let $B\subseteq \monster$ not necessarily be small. A \textit{$B$-definable cell} is a set $\gamma\subseteq\monster^x$ of the form $\psi_1(x_{\neq 1},c_1)\wedge \cdots\wedge \psi_{k+1}(x_{\neq k+1},c_{k+1})$, where $c:=(c_1, ..., c_{k+1})\in B^{|z_1|+\cdots+|z_{k+1}|}$; write $\gamma_c$ for this set. Write $\mathcal{G}_B$ for the set of all $B$-definable cells.

    For $a\in \monster^{x_{\neq k+1}}$ and $c^{(1)}, ..., c^{(N)}\in \monster^{|z_1|+\cdots+|z_{k+1}|}$, let $F_{a,c^{(1)}, ..., c^{(N)}}$ be the set
    \[\{b\in \monster^{x_{k+1}}: a\models (x_{\neq k+1},b)\in\gamma_{c^{(j)}}\vdash \varphi(x_{\neq k+1};b)\leftrightarrow \varphi(a;b)\text{ for some }j\in [N]\}.\]

    A tuple $\Gamma=(\gamma_{c^{(1)}}, ..., \gamma_{c^{(N)}})\in \mathcal{G}_B^N$ of $B$-definable cells is \textit{$B$-complete} if there is $a\in \monster^{x_{\neq k+1}}$ such that $B^{x_{k+1}}\subseteq F_{a,c^{(1)}, ..., c^{(N)}}$, in which case we say that $\Gamma$ is $B$-complete \textit{with respect to} $a$. For $\Gamma=(\gamma_{c^{(1)}}, ..., \gamma_{c^{(N)}})\in \mathcal{G}^N_B$, we write $\gamma\in \Gamma$ to mean $\gamma=\gamma_{c^{(j)}}$ for some $j\in [N]$.
\end{defn}

\begin{remark}\label{rmkcompletecell}
    Since $(\psi_i(x_{\neq i},z_i): i\in [k+1])$ is a $k$-strong honest definition for $\varphi$ of degree $N$, for all $a\in \monster$ and $B\subseteq \monster$ with $2\leq |B|<\infty$, there are $c^{(1)}, ..., c^{(N)}\in B$ such that $(\gamma_{c^{(1)}}, ..., \gamma_{c^{(N)}})\in\mathcal{G}_B^N$ is $B$-complete with respect to $a$.
\end{remark}

We prove the following `cutting lemma'.
\begin{prop}\label{cuttinglemma}
    For all $r\in\mathbb{R}_{\geq 1}$, there is a finite set $B\subseteq M$ with $|B|\geq 2$ and $|B|=O_{\varphi,\psi_1, ..., \psi_{k+1},N}(r^2\log 2r)$ such that the following holds.
    
    Let $\nu(x_{k+1})$ be a Keisler measure, generically stable over $M$. For $a\in \monster$ and $c^{(1)}, ..., c^{(N)}\in B$, if $(\gamma_{c^{(1)}}, ..., \gamma_{c^{(N)}})\in\mathcal{G}_B^N$ is $B$-complete with respect to $a$, then $\nu(F_{a,c^{(1)}, ..., c^{(N)}})\geq 1-\frac{1}{r}$.
\end{prop}
\begin{proof}
    Let $d:=|x_{k+1}|$. Applying Theorem \ref{epsilonapprox} to the definable family $\mathcal{F}:=\{F_{a,c^{(1)}, ..., c^{(N)}}: a, c^{(1)}, ..., c^{(N)}\in \monster\}$, there is a multiset $S\subseteq M^d$ of size $O_{\varphi,\psi_1, ..., \psi_{k+1}, N}(r^2\log 2r)$, such that, for all $F\in\mathcal{F}$, $|\nu(F)-\text{Av}(S;F)|\leq \frac{1}{r}$.

    Choose $B\subseteq M$ such that $B$ contains all of singletons appearing in $S$ and $2\leq |B|\leq d|S|+2$. If $(\gamma_{c^{(1)}}, ..., \gamma_{c^{(N)}})\in\mathcal{G}_B^N$ is $B$-complete with respect to $a$, then $S\subseteq B^d\subseteq F_{a,c^{(1)}, ..., c^{(N)}}$ as sets, that is, $\text{Av}(S;F_{a,c^{(1)}, ..., c^{(N)}})=1$, and so $\nu(F_{a,c^{(1)}, ..., c^{(N)}})\geq 1-\frac{1}{r}$.
\end{proof}
\begin{defn}
    Let $Z\subseteq \monster^x$, $a\in \monster^{x_{\neq k+1}}$, and $b\in \monster^{x_{k+1}}$. Write $Z|^a:=\{b'\in \monster^{x_{k+1}}: (a,b')\in Z\}$ and $Z|_b:=\{a'\in \monster^{x_{\neq k+1}}: (a',b)\in Z\}$.
\end{defn}
    We are now ready to prove Theorem \ref{kSHDreglemma}.
\begin{proof}[Proof of Theorem \ref{kSHDreglemma}]
    Apply Proposition \ref{cuttinglemma} with $r:=\frac{1}{\delta}\geq 1$ to obtain $B\subseteq M$ with $|B|\geq 2$ and $|B|=O_{\varphi,\psi_1, ..., \psi_{k+1},N}(\delta^{-2}\log 2\delta^{-1})=O_{\varphi,\psi_1, ..., \psi_{k+1},N}(\delta^{-3})$. We have $|\mathcal{G}_B|\leq |B|^l$ for $l:=|z_1|+\cdots+|z_{k+1}|$.
    
    For $\gamma\in\mathcal{G}_B$, let $D_\gamma=\{b\in \monster^{x_{k+1}}: \gamma|_b\subseteq \varphi(x_{\neq k+1};b)\text{ or }\gamma|_b\subseteq \neg\varphi(x_{\neq k+1};b)\}$. Let $G:=\bigvee_{\gamma\in \mathcal{G}_B}\gamma\wedge D_\gamma$. We claim that $\omega(G)\geq 1-\delta$. It suffices to show that, for all $a\in \monster^{x_{\neq k+1}}$, $\nu\left(G|^a\right)\geq 1-\delta$.
    
    Fix $a\in \monster^{x_{\neq k+1}}$. By Remark \ref{rmkcompletecell}, there is $\Gamma=(\gamma_{c^{(1)}}, ..., \gamma_{c^{(N)}})\in \mathcal{G}_B^N$ which is $B$-complete with respect to $a$. It suffices to show that $(\bigvee_{\gamma\in \Gamma}\gamma\wedge D_\gamma)|^a\supseteq F_{a,c^{(1)}, ..., c^{(N)}}$, as $\nu(F_{a,c^{(1)}, ..., c^{(N)}})\geq 1-\delta$ by our choice of $B\subseteq M$ from Proposition \ref{cuttinglemma}. So, suppose $b\in F_{a,c^{(1)}, ..., c^{(N)}}$. Then, there is $\gamma\in \Gamma$ such that
    \[a\models (x_{\neq k+1}, b)\in \gamma\vdash \varphi(x_{\neq k+1}; b)\leftrightarrow \varphi(a;b).\]
    Thus, $(a,b)\in\gamma$. Also, writing $\varphi^1:=\varphi$ and $\varphi^0:=\neg\varphi$, we have $\gamma|_b\subseteq \varphi^\sigma(x_{\neq k+1};b)$ for the unique $\sigma\in\{0,1\}$ such that $\monster\models \varphi^\sigma(a;b)$, so $b\in D_\gamma$ as required. We have shown that $\omega(G)\geq 1-\delta$.

    For $\gamma\in\mathcal{G}_B$ and $\sigma\in\{0,1\}$, let $D_\gamma^\sigma:=\{b\in\monster^{x_{k+1}}: \gamma|_b\subseteq \varphi^\sigma(x_{\neq k+1};b)\}$, so that $D_\gamma=D_\gamma^0\sqcup D_\gamma^1$. Let the partition $\mathcal{P}_1$ of $\monster^{x_{\neq 1}}$ be the set of Boolean atoms\footnote{Given a set $X$ and a collection $\mathcal{S}=\{S_1, ..., S_n\}$ of subsets of $X$, a \textit{Boolean atom} of $\mathcal{S}$ is a non-empty set of the form $\bigcap_{i\in [n]}S_i^{\varepsilon(i)}$ for some $\varepsilon: [n]\to \{0,1\}$, where $S^1:=S$ and $S^0:=X\setminus S$ for $S\subseteq X$.} of $\{\psi_1(x_{\neq 1},c_1): c_1\in B\}\cup \{D_\gamma^\sigma: \gamma\in\mathcal{G}_B, \sigma\in \{0,1\}\}$, where $D_\gamma^\sigma$ is identified with the definable set $\{(x_2, ..., x_{k+1}): x_{k+1}\in D_\gamma^\sigma\}$. For $i\in [k+1]\setminus \{1\}$, let the partition $\mathcal{P}_i$ of $\monster^{x_{\neq i}}$ be the set of Boolean atoms of $\{\psi_i(x_{\neq i},c_i): c_i\in B\}$. Since $|\mathcal{G}_B|\leq |B|^l$ and $M$ is NIP, for all $i\in [k+1]$ we have that
    \begin{equation}\label{eqntypecounting}
        |\mathcal{P}_i|\leq \text{poly}_{\varphi,\psi_1, ..., \psi_{k+1},N}(|B|)\leq \text{poly}_{\varphi,\psi_1, ..., \psi_{k+1},N}(\delta^{-1}).
    \end{equation}

    It is clear that there are $L$-formulas $\theta_i(x_{\neq i},z_i)$ for $i\in [k+1]$, which are functions of $\varphi, \psi_1, ..., \psi_{k+1}, N$, and $\delta$, such that (i) holds. To see that (ii) holds, recall that $\omega(G)\geq 1-\delta$ for $G=\bigvee_{\gamma\in\mathcal{G}_B}\gamma\wedge D_\gamma=\bigvee_{\gamma\in\mathcal{G}_B}\bigvee_{\sigma\in\{0,1\}}\gamma\wedge D_\gamma^\sigma$. For all $\gamma\in\mathcal{G}_B$ and $\sigma\in\{0,1\}$, $\gamma\wedge D_\gamma^\sigma$ is $\varphi$-homogeneous (indeed, $\gamma\wedge D_\gamma^\sigma\subseteq \varphi^\sigma(x)$), and by the definition of Boolean atoms, $\gamma\wedge D_\gamma^\sigma$ is a union of sets of the form $P_1\wedge \cdots \wedge P_{k+1}$ where $P_i\in\mathcal{P}_i$. Therefore, the union of all $\varphi$-homogeneous sets of the form $P_1\wedge\cdots\wedge P_{k+1}$ contains $G$, and has $\omega$-measure at least $1-\delta$. This shows that (ii) holds.
\end{proof}
    As mentioned before, our proof strategy follows that of \cite[Theorem 5.8]{regularitylemma}. There, they also prove a cutting lemma, which they use to prove that the hypergraph satisfies the definable `strong Erd\H{o}s--Hajnal property' \cite[Proposition 4.4]{regularitylemma}, before bootstrapping it into a regularity lemma. Let us state an abridged form of \cite[Proposition 4.4]{regularitylemma}.
    \begin{prop}
        Let $\chi(x_1, ..., x_{k+1})$ be a relation definable in a distal structure $\mathcal{M}$, and suppose $|x_1|=\cdots=|x_{k+1}|$. For all $\alpha\in (0,1]$, there are $\varepsilon>0$ and $\theta_i(x_i, z_i)\in L$ for $i\in [k+1]$ such that the following holds.
        
        Let $\nu(x_{k+1})$ be a Keisler measure, generically stable over $\mathcal{M}$. If $\nu^{(k+1)}(\chi)\geq \alpha$, then there are $c_1, ..., c_{k+1}\in \mathcal{M}$ such that $\bigwedge_{i\in [k+1]}\theta_i(x_i, c_i)$ is contained in $\chi$ and has $\nu^{(k+1)}$-measure at least $\varepsilon$.
    \end{prop}

    In our proof, we observe that such an intermediate step is not necessary: once we have a cutting lemma, we can directly define the appropriate partitions to give us the desired regularity lemma. Note that an analogue of the definable strong Erd\H{o}s--Hajnal property for our relation $\varphi$ can then easily be \textit{deduced} from our regularity lemma.
    \begin{cor}
        ($T$ is NIP, $\varphi$ as before.) Suppose $|x_1|=\cdots=|x_{k+1}|$. For all $\alpha\in (0,1]$, there are $\varepsilon>0$ and $\theta_i(x_{\neq i}, z_i)\in L$ for $i\in [k+1]$ such that the following holds.
        
        Let $\nu(x_{k+1})$ be a Keisler measure, generically stable over $M$. If $\nu^{(k+1)}(\varphi)\geq \alpha$, then there are $c_1, ..., c_{k+1}\in M$ such that $\bigwedge_{i\in [k+1]}\theta_i(x_{\neq i}, c_i)$ is contained in $\varphi$ and has $\nu^{(k+1)}$-measure at least $\varepsilon$.
    \end{cor}
    \begin{proof}
        Applying Theorem \ref{kSHDreglemma} with $\delta=\alpha/2$, we have uniformly $M$-definable partitions $\mathcal{P}_1, ..., \mathcal{P}_{k+1}$ with $|\mathcal{P}_i|\leq K\leq \text{poly}_{\varphi,\psi_1, ..., \psi_{k+1},N}(\delta^{-1})$ for all $i\in [k+1]$, such that
        \[\sum \nu^{(k+1)}(P_1\wedge\cdots \wedge P_{k+1})\geq \alpha-\delta= \alpha/2,\]
        where the sum ranges over all $(P_1, ..., P_{k+1})\in \mathcal{P}_1\times\cdots\times \mathcal{P}_{k+1}$ such that $P_1\wedge \cdots \wedge P_{k+1}\subseteq \varphi$. One of these tuples $(P_1, ..., P_{k+1})$ is such that
        \[\nu^{(k+1)}(P_1\wedge\cdots\wedge P_{k+1})\geq \frac{\alpha}{2|\mathcal{P}_1|\cdots|\mathcal{P}_{k+1}|}\geq \frac{\alpha}{2K^{k+1}}.\qedhere\]
    \end{proof}
\subsection{Main result}
In this subsection, we prove a more general version of Theorem \ref{kSHDreglemma}. Throughout this subsection, we fix an NIP $L$-theory $T$ and models $M,\monster\models T$ with $\monster$ sufficiently saturated and $M$ small.

We can make Theorem \ref{kSHDreglemma} uniform, in the sense that if $\varphi(x_1, ..., x_{k+1})=\varphi'(x_1, ..., x_{k+1}, e)$ for some $e\in M$, then $\theta_i$ and $K$ can be chosen independently of $e$. The following theorem is the most general formulation of our regularity lemma in this paper.

\begin{theorem}\label{kSHDreglemmauniform}
    ($T$ is NIP.) Let $\varphi'(x_1, ..., x_k; (x_{k+1}, u))\in L$ have $k$-strong honest definition $(\psi_1,$ ..., $\psi_{k+1})$ of degree $N$. For all $\delta \in (0,1]$, there are $\theta_i(x_{\neq i},z_i)\in L$ for $i\in [k+1]$ and a natural number $K\leq \text{poly}_{\varphi',\psi_1, ..., \psi_{k+1}, N}(\delta^{-1})$, such that the following holds.

    Let $\varphi(x_1, ..., x_{k+1}):=\varphi'(x_1, ..., x_k; (x_{k+1}, e))$ for some $e\in M$. Let $\mu(x_{\neq k+1})$ and $\nu(x_{k+1})$ be Keisler measures, with $\nu(x_{k+1})$ generically stable over $M$, and let $\omega(x):=\nu(x_{k+1})\otimes\mu(x_{\neq k+1})$. Then there are partitions $\mathcal{P}_i$ of $\monster^{x_{\neq i}}$ for $i\in [k+1]$, each of size at most $K$, such that:
    \begin{enumerate}[(i)]
        \item for all $i\in [k+1]$ and $P_i\in\mathcal{P}_i$, there is $c_i\in M^{z_i}$ such that $P_i=\theta_i(x_{\neq i},c_i)$;
        \item $\sum \omega(P_1\wedge \cdots\wedge P_{k+1})\geq 1-\delta$, where the sum ranges over all $(P_1, ..., P_{k+1})\in\mathcal{P}_1\times\cdots\times \mathcal{P}_{k+1}$ such that $P_1\wedge \cdots \wedge P_{k+1}$ is $\varphi$-homogeneous.
    \end{enumerate}
\end{theorem}
\begin{proof}
    Let $\nu'(x_{k+1},u)$ be the Keisler measure given by $\nu'(\chi(x_{k+1},u)):=\nu(\chi(x_{k+1}, e))$ for all $\chi(x_{k+1}, u)\in L(\monster)$. Using Theorem \ref{epsilonapprox}, it is an easy exercise to show that $\nu'$ is generically stable over $M$. Applying Theorem \ref{kSHDreglemma} to $\varphi'$ with the Keisler measures $\mu(x_{\neq k+1})$ and $\nu'(x_{k+1},u)$, we are done.
\end{proof}
We would like to remove the NIP assumption from our main theorem, possibly at the cost of weakening the statement.
\begin{problem}\label{qnnoNIP2}
    Suppose $T$ is not necessarily NIP (perhaps $T$ is $\text{NIP}_k$). Must Theorem \ref{kSHDreglemmauniform} hold? If not, can one still obtain a (weaker) homogeneous regularity lemma for $\varphi$?
\end{problem}
The NIP assumption was used in the proof of Theorem \ref{kSHDreglemmauniform} in the following instances:
\begin{enumerate}[(i)]
    \item to define the product of Keisler measures (Definition \ref{defnproductmeasure});
    \item to invoke an $\varepsilon$-approximation result (Theorem \ref{epsilonapprox});
    \item to ensure a polynomial bound on the size of the partitions, via a type-counting argument (equation (\ref{eqntypecounting}) in the proof of Theorem \ref{kSHDreglemma}).
\end{enumerate}
Of these, instances (i) and (iii) are relatively easy to circumvent (at a small cost) --- the former by restricting our attention to finite counting measures, and the latter by accepting an inferior bound on the size of the partitions. To circumvent instance (ii) would be a more substantial undertaking; one possible approach is to generalise the theory of $\varepsilon$-approximations (and, more generally, of sampling) in the NIP setting to the $\text{NIP}_k$ setting.

We record the special case of Theorem \ref{kSHDreglemmauniform} where $|x_1|=\cdots=|x_{k+1}|=:d$, and we state the regularity lemma for a hypergraph defined on $\bigsqcup_{i=1}^{k+1}V$ for some $\monster$-definable $V\subseteq \monster^d$.
\begin{cor}\label{kSHDregcor}
    ($T$ is NIP.) Let $\varphi'(x_1, ..., x_k; (x_{k+1}, u))\in L$ have $k$-strong honest definition $(\psi_1, ..., \psi_{k+1})$ of degree $N$, and suppose $|x_1|=\cdots=|x_{k+1}|=:d$. For all $\delta \in (0,1]$, there are $\theta(x_1, ...,  x_k, z)\in L$ and a natural number $K\leq \text{poly}_{\varphi',\psi_1, ..., \psi_{k+1}, N}(\delta^{-1})$ such that the following holds.

    Let $\varphi(x_1, ..., x_{k+1}):=\varphi'(x_1, ..., x_k; (x_{k+1}, e))$ for some $e\in M$. Let $V\subseteq \monster^d$ be $\monster$-definable, and let $\nu(x_{k+1})$ be a global measure, generically stable over $M$. Then there is a partition $\mathcal{P}$ of $V^k$ of size at most $K$ such that:
    \begin{enumerate}[(i)]
        \item for all $P\in\mathcal{P}$, there is $c\in M^z$ such that $P=\theta(x_1, ..., x_k, c)\cap V^k$;
        \item $\sum \nu^{(k+1)}(P_1\wedge \cdots\wedge P_{k+1})\geq (1-\delta) \nu(V)^{k+1}$, where the sum ranges over all $(P_1, ..., P_{k+1})\in \mathcal{P}^{k+1}$ such that $P_1\wedge \cdots \wedge P_{k+1}$ is $\varphi$-homogeneous.
    \end{enumerate}
\end{cor}
\begin{proof}
    Without loss of generality, suppose $\nu(V)>0$. By Proposition \ref{relativisation}, we have that $\nu|_V$, the relativisation of $\nu$ to $V$, is generically stable over $M$. Applying Theorem \ref{kSHDreglemmauniform} with $\nu|_V(x_{k+1})$ and $\mu(x_1, ..., x_k):=(\nu|_V)^{(k)}(x_1, ..., x_k)$, we are done. (Note that the formulas $\theta_1, ..., \theta_{k+1}$ given by Theorem \ref{kSHDreglemmauniform} can easily be coded into one formula $\theta$.)
\end{proof}
Since finite counting measures are generically stable (see the discussion after Theorem \ref{epsilonapprox}), we have the following statement for finite hypergraphs. We formulate this in a manner more consistent with our earlier combinatorial discourse.
\begin{cor}\label{kSHDregcorfinite}
    ($T$ is NIP.) Let $\varphi'(x_1, ..., x_k; (x_{k+1}, u))\in L$ have $k$-strong honest definition $(\psi_1, ..., \psi_{k+1})$ of degree $N$, and suppose $|x_1|=\cdots=|x_{k+1}|=:d$. For all $\delta \in (0,1]$, there are $\theta(x_1, ...,  x_k, z)\in L$ and a natural number $K\leq \text{poly}_{\varphi',\psi_1, ..., \psi_{k+1},N}(\delta^{-1})$ such that the following holds.

    Let $\varphi(x_1, ..., x_{k+1}):=\varphi'(x_1, ..., x_k; (x_{k+1}, e))$ for some $e\in M$, and let $V\subseteq M^d$ be finite. Then there is a partition $\mathcal{P}$ of $V^k$ of size at most $K$ such that:
    \begin{enumerate}[(i)]
        \item for all $P\in\mathcal{P}$, there is $c\in M^z$ such that $P=\theta(x_1, ..., x_k, c)\cap V^k$;
        \item the induced partition $\mathcal{D}$ of $V^{k+1}$, given by
        \begin{align*}
        &\left\lbrace P_1\wedge\cdots\wedge P_{k+1}: P_i\in\mathcal{P}\text{ for all }i\in [k+1]\right\rbrace\\
            &=\left\lbrace \left\lbrace w=(w_1, ..., w_{k+1})\in V^{k+1}: w_{\neq i}\in P_i\text{ for all } i\in [k+1]\right\rbrace: P_1, ..., P_{k+1}\in\mathcal{P}\right\rbrace,
        \end{align*}
    is such that $\sum_{D\in\mathcal{D}\text{ }\varphi\text{-homogeneous}} |D|\geq (1-\delta) |V|^{k+1}$.
    \end{enumerate}
\end{cor}
\subsection{A version for counting}\label{subseccounting}
Throughout this subsection, fix an NIP $L$-theory $T$ and models $M,\monster\models T$ with $\monster$ sufficiently saturated and $M$ small.

As mentioned in Section \ref{sechypreg}, hypergraph regularity lemmas are often used in conjunction with counting lemmas. The regularity lemmas we have proven for $(k+1)$-graphs $(\bigsqcup_{i=1}^{k+1}V, E)$ give a partition of $V^k$ that induces a partition of $V^{k+1}$ into simplicial complexes, most of which are homogeneous. However, in order to have an associated counting lemma, the parts of the partition of $V^k$ should themselves be quasirandom as $k$-graphs. In this subsection, we describe how we can refine our regularity lemmas \textit{a posteriori} to achieve this. Since we expect this to be mainly of interest to finite combinatorialists, we will only focus on finite hypergraphs, but our methods also apply in the context of Keisler measures considered above (as long as $\omega(x_1, ..., x_{k+1})$ is given by a product of generically stable measures $\mu_1(x_1)$, ..., $\mu_{k+1}(x_{k+1})$). To simplify the notation, we will specialise to the case where $k=2$ (thus working towards a regularity lemma for 3-graphs), which already contains all the important ideas.

We appeal to the following regularity lemma for graphs definable in an NIP theory from the literature. A version of this was first proven by Alon, Fischer, and Newman \cite{alonfischernewman} and Lov\'asz and Szegedy \cite{lovaszszegedy} for finite graphs of bounded VC-dimension. Of course, their result readily applies to graphs definable in an NIP theory, but the following statement, due to Chernikov and Starchenko \cite{definablenipregularity}, contains extra definable data.
\patchcmd{\thmhead}{(#3)}{#3}{}{}
\begin{theorem}[\cite{definablenipregularity}]\label{nipregularity}
    ($T$ is NIP.) Let $\theta(x_1, x_2, z)\in L$ with $|x_1|=|x_2|=d$. For all $\delta \in (0,1]$, there is $\eta(x_1,w)\in L$ and a natural number $K\leq \text{poly}_\theta(\delta^{-1})$ such that the following holds.

    Let $\theta_e(x_1, x_2):=\theta(x_1, x_2, e)$ for some $e\in M$, and let $V\subseteq M^d$ be finite. Then there is a partition $\mathcal{Q}$ of $V$ of size at most $K$ such that:
    \begin{enumerate}[(i)]
        \item for all $Q\in\mathcal{Q}$, there is $r\in M^w$ such that $Q=\eta(x_1,r)\cap V$;
        \item $\sum |Q_1\times Q_2|\geq (1-\delta)|V|^2$, where the sum ranges over all $(Q_1, Q_2)\in \mathcal{Q}^2$ such that the density of $\theta_e(Q_1, Q_2)$ is $\delta$-close to $0$ or $1$, that is,
        \begin{equation}\label{eqnepsilonhomogeneity}
            \frac{\abs{\theta_e\cap Q_1\times Q_2}}{\abs{Q_1\times Q_2}}\in [0,\delta]\cup [1-\delta, 1].
        \end{equation}
    \end{enumerate}    
\end{theorem}
\patchcmd{\thmhead}{#3}{(#3)}{}{}
The notion of quasirandomness captured by (\ref{eqnepsilonhomogeneity}) is very strong (although weaker than homogeneity); where this holds, we will say that $Q_1\times Q_2$ is \textit{$\delta$-almost $\theta_e$-homogeneous}.
\begin{remark}
    Theorem \ref{nipregularity} is a regularity lemma for the graph $\theta_e(x_1, x_2)$. The analogous result for hypergraphs also holds, as proven in \cite{definablenipregularity}, and can be used instead of Theorem \ref{nipregularity} in the case where $k>2$.
\end{remark}

For the rest of this subsection, we put ourselves in the context of Corollary \ref{kSHDregcorfinite}, fixing the following.
\begin{itemize}
    \item A formula $\varphi(x_1, x_2; x_3)\in L$ with $|x_1|=|x_2|=|x_3|=:d$; write $x:=(x_1, x_2, x_3)$.
    \item A $2$-strong honest definition $(\psi_i(x_{\neq i},z_i): i\in [3])$ for $\varphi$ of degree $N$.
\end{itemize}
Note that we have dropped the parameter $u$ from Corollary \ref{kSHDregcorfinite} for simplicity.

In Corollary \ref{kSHDregcorfinite}, we have a partition $\mathcal{P}$ of $V^2$, inducing simplicial complexes of the form
\[P_1\wedge P_2\wedge P_3=\{w=(w_1, w_2,w_3)\in V^3: w_{\neq i}\in P_i\text{ for all }i\in [3]\}\]
for $P_1, P_2, P_3\in\mathcal{P}$, such that most of these simplicial complexes are $\varphi$-homogeneous. Rewriting $V^3$ as $V_1\times V_2\times V_3$, $P_1\wedge P_2 \wedge P_3$ is the simplicial complex comprising the triangles in $P_1(V_2, V_3)\cup P_2(V_1, V_3)\cup P_3(V_1, V_2)$, where $P_1$ is viewed as a graph on $V_2\times V_3$, and so on.

In order to make this regularity lemma amenable to counting lemmas, we refine these simplicial complexes by intersecting them with boxes $Q_1\times Q_2\times Q_3$ obtained using Theorem \ref{nipregularity}. The resulting simplicial complex $(P_1\wedge P_2 \wedge P_3)\cap (Q_1\times Q_2\times Q_3)$ then comprises the triangles in $P_1(Q_2, Q_3)\cup P_2(Q_1, Q_3)\cup P_3(Q_1, Q_2)$. Our new notion of quasirandomness for these simplicial complexes is as follows.

\begin{defn}
    For $i\in [3]$, let $P_i\subseteq \monster^{2d}$ and $Q_i\subseteq \monster^d$ be finite. For $\delta>0$, say that the simplicial complex $(P_1\wedge P_2 \wedge P_3)\cap (Q_1\times Q_2\times Q_3)$ is \textit{$(\varphi,\delta)$-homogeneous} if it is $\varphi$-homogeneous and, for all $i\in [3]$, $\prod_{j\in [3], j\neq i}Q_j$ is $\delta$-almost $P_i$-homogeneous.
\end{defn}

Combining Corollary \ref{kSHDregcorfinite} with Theorem \ref{nipregularity}, we obtain the following.
\begin{theorem}
    ($T$ is NIP.) For all $\delta\in (0,1]$, there are $\theta(x_1, x_2, z), \eta(x_1, w)\in L$ and a natural number $K\leq \text{poly}_{\varphi,\psi_1, \psi_2, \psi_3,N}(\delta^{-1})$ such that the following holds.

    Let $V\subseteq M^d$ be finite. Then there are partitions $\mathcal{P}$ of $V^2$ and $\mathcal{Q}$ of $V$ of size at most $K$ such that:
    \begin{enumerate}[(i)]
        \item for all $P\in\mathcal{P}$, there is $c\in M^z$ such that $P=\theta(x_1, x_2, c)\cap V^2$;
        \item for all $Q\in\mathcal{Q}$, there is $r\in M^w$ such that $Q=\eta(x_1, r)\cap V$;
        \item the induced partition $\mathcal{D}$ of $V^3$, given by
        \begin{align*}
            &\left\lbrace (P_1\wedge P_2 \wedge P_3)\cap (Q_1\times Q_2\times Q_3): P_i\in\mathcal{P}, Q_i\in\mathcal{Q}\text{ for all }i\in [3]\right\rbrace\\
            &=\left\lbrace \left\lbrace w=(w_1, w_2, w_3)\in V^3: w_{\neq i}\in P_i, w_i\in Q_i\text{ for all } i\in [3]\right\rbrace: P_1, P_2, P_3\in\mathcal{P}, Q_1, Q_2, Q_3\in\mathcal{Q}\right\rbrace,
            \end{align*}
    is such that $\sum_{D\in\mathcal{D}\text{ }(\varphi,\delta)\text{-homogeneous}} |D|\geq (1-\delta) |V|^3$.
    \end{enumerate}
\end{theorem}
\begin{proof}
    First apply Corollary \ref{kSHDregcorfinite} to $\varphi'=\varphi$ with $\delta/2$ in place of $\delta$ to obtain $\theta(x_1, x_2, z)\in L$, a natural number $K\leq \text{poly}_{\varphi,\psi_1, \psi_2, \psi_3,N}(\delta^{-1})$, and a partition $\mathcal{P}$ of $V^2$. For each $P\in \mathcal{P}$, let $c_P\in M^z$ be such that $P=\theta(x_1, x_2, c_P)\cap V$. Applying Theorem \ref{nipregularity} to $\theta(x_1, x_2, z)$ with $\delta':=\delta(6K^3)^{-1}$ in place of $\delta$, we obtain $\eta(x_1,w)\in L$ and a natural number $K'\leq \text{poly}_\theta((\delta')^{-1})\leq \text{poly}_{\varphi,\psi_1, \psi_2, \psi_3,N}(\delta^{-1})$ such that the following holds. For each $P\in\mathcal{P}$, there is a partition $\mathcal{Q}_P$ of $V$ of size at most $K'$, such that:
    \begin{enumerate}[(i)]
        \item for all $Q\in \mathcal{Q}_P$, there is $r\in M^w$ such that $Q=\eta(x_1, r)\cap V$;
        \item $\sum |Q_1\times Q_2|\geq (1-\delta')|V|^2$, where the sum ranges over all $(Q_1, Q_2)\in \mathcal{Q}_P^2$ such that $Q_1\times Q_2$ is $\delta'$-almost $P$-homogeneous (since $P=\theta(x_1, x_2, c_P)$).
    \end{enumerate}

    Let $R\subseteq M^w$ be minimal such that, for every $P\in\mathcal{P}$ and $Q\in\mathcal{Q}_P$, there is $r\in R$ such that $Q=\eta(x_1, r)\cap V$; then, $|R|\leq KK'\leq \text{poly}_{\varphi,\psi_1, \psi_2, \psi_3,N}(\delta^{-1})$. Let $\mathcal{Q}$ be the set of Boolean atoms of $\{\eta(x_1, r)\cap V: r\in R\}$. Since $T$ is NIP, $|\mathcal{Q}|\leq \poly_\eta (|R|)\leq \text{poly}_{\varphi,\psi_1, \psi_2, \psi_3,N}(\delta^{-1})$. Observe that for all $P\in\mathcal{P}$ we have $\sum |Q_1\times Q_2|\leq \delta'|V|^2$, where the sum ranges over all $(Q_1, Q_2)\in \mathcal{Q}^2$ such that $Q_1\times Q_2$ is not $\delta'$-almost $P$-homogeneous.

    Let $\mathcal{D}$ be the partition of $V^3$ induced by $\mathcal{P}$ and $\mathcal{Q}$. Observe that
    \[\sum_{D\in\mathcal{D}\text{ not }(\varphi,\delta)\text{-homogeneous}}|D|\leq I_1+I_2,\]
    where
    \begin{align*}
        I_1:=&\sum_{\substack{(P_1, P_2, P_3)\in\mathcal{P}^3\\P_1\wedge P_2 \wedge P_3\text{ not }\varphi\text{-homogeneous}}}|P_1\wedge P_2 \wedge P_3|,\\
        I_2:=&\sum_{(P_1, P_2, P_3)\in\mathcal{P}^3}\sum_{\substack{(Q_1, Q_2, Q_3)\in\mathcal{Q}^3, \text{ there is }i\in [3]\text{ with}\\\prod_{j\in [3], j\neq i}Q_j\text{ not }\delta'\text{-almost }P_i\text{-homogeneous}}}|Q_1\times Q_2\times Q_3|.
    \end{align*}

    By definition of $\mathcal{P}$ obtained from Corollary \ref{kSHDregcorfinite}, we have $I_1\leq (\delta/2)|V|^3$. For $I_2$, fix $P_1, P_2, P_3\in\mathcal{P}$ and investigate the inner sum. For each $i\in [3]$, we have
    \begin{align*}
        &\sum_{\substack{(Q_1, Q_2, Q_3)\in\mathcal{Q}^3,\\ \prod_{j\in [3], j\neq i}Q_j\text{ not }\delta'\text{-almost }P_i\text{-homogeneous}}}|Q_1\times Q_2\times Q_3|\\
        &= |V|\sum_{\substack{(Q_1, Q_2)\in\mathcal{Q}^2,\\Q_1\times Q_2\text{ not }\delta'\text{-almost }P_i\text{-homogeneous}}}|Q_1\times Q_2|\\
        &\leq \delta' |V|^3.
    \end{align*}
    
    Thus, $I_2\leq 3\delta'|V|^3K^3=(\delta/2)|V|^3$, so $\sum_{D\in\mathcal{D}\text{ not }(\varphi,\delta)\text{-homogeneous}}|D|\leq \delta |V|^3$ as required.
\end{proof}
We sketch how a tetrahedron counting lemma may be obtained for $(\varphi,\delta)$-homogeneity. Let $V\subseteq M^d$ be finite, $Q_1, Q_2, Q_3, Q_4\subseteq V$, and $P_{ij}\subseteq V^2$ for all $i,j\in [4]$ with $i<j$. For all $i,j,k\in [4]$ with $i<j<k$, define the simplicial complex $\Sigma_{ijk}:=(P_{jk}\wedge P_{ik}\wedge P_{ij})\cap (Q_i\times Q_j\times Q_k)$; we view this as the 3-partite 3-graph with vertex set $Q_i\sqcup Q_j\sqcup Q_k$, whose hyperedges are precisely the triangles in the graph $P_{jk}(Q_j,Q_k)\cup P_{ik}(Q_i,Q_k)\cup P_{ij}(Q_i, Q_j)$. Let $W$ be the 4-partite 3-graph given by $\Sigma_{123}\cup \Sigma_{124}\cup \Sigma_{134}\cup \Sigma_{234}$. We wish to show that the number of tetrahedra in $\varphi\cap W$ is approximately as `expected' if $\Sigma_{ijk}$ is $(\varphi,\delta)$-homogeneous for all $i,j,k$.

What is the `expected' number? For all $i,j\in [4]$ with $i<j$, let $d_{ij}$ be the density of $P_{ij}(Q_i, Q_j)$, and for all $i,j,k\in [4]$ with $i<j<k$, let $d_{ijk}$ be the relative density of $\varphi$ in $\Sigma_{ijk}$. If $\varphi$ were a random hyperedge relation, the expected number of tetrahedra in $\varphi\cap W$ would be $N_{\exp}:=(\prod_{ij}d_{ij}\prod_{ijk}d_{ijk})N_0$, where $N_0:=\prod_i|Q_i|$. Now, for all $i,j,k$, the simplicial complex $\Sigma_{ijk}$ is $\varphi$-homogeneous, so $d_{ijk}\in\{0,1\}$. If, for some $i,j,k$, $d_{ijk}=0$, then $\varphi\cap W$ contains no tetrahedra and $N_{\exp}=0$, so $\varphi\cap W$ contains exactly the correct number of tetrahedra. Suppose instead that $d_{ijk}=1$ for all $i,j,k$, so that $N_{\exp}=(\prod_{i,j}d_{ij})N_0$. For all $i,j$, $d_{ij}\in [0,\delta]\cup [1-\delta, 1]$. If $d_{ij}\in [0,\delta]$ for some $i,j$, then $N_{\exp}\leq\delta N_0$, and the number of tetrahedra in $\varphi\cap W$ is at most $\delta N_0$ (since a tetrahedron $(v_1, ..., v_4)\in Q_1\timesdots Q_4$ must have $v_{ij}\in P_{ij}$), thus differing from $N_{\exp}$ by at most $\delta N_0$. If $d_{ij}\in [1-\delta, 1]$ for all $i,j$, then $N_{\exp}\geq (1-6\delta)N_0$, and the number of tetrahedra in $\varphi \cap W$ is at least $(1-6\delta)N_0$ (since a non-tetrahedron $(v_1, ..., v_4)\in Q_1\timesdots Q_4$ must have $v_{ij}\not\in P_{ij}$ for some $i,j$), thus differing from $N_{\exp}$ by at most $6\delta N_0$. This shows that $\varphi\cap W$ contains approximately the expected number of tetrahedra.
\subsection{Future work: recovering strong $k$-distality}
By Theorem \ref{uniformkSHD}, the regularity lemma in Corollary \ref{kSHDregcorfinite} applies to all relations definable in an NIP strongly $k$-distal structure. We can ask if every relation $\varphi$ on a set $M$ satisfying this regularity lemma (without the definable data) is such that $(M,\varphi)$ admits an expansion that is NIP strongly $k$-distal.
\begin{defn}
    Let $\varphi(x_1, ..., x_{k+1})$ be a relation on a set $M$. Say that $\varphi$ \textit{satisfies the NIP strongly $k$-distal regularity lemma} if the following holds.

    For all $\delta \in (0,1]$, there is a natural number $K\leq \text{poly}_{\varphi}(\delta^{-1})$ such that for all finite $V\subseteq M^d$, there is a partition $\mathcal{P}$ of $V^k$ of size at most $K$ inducing a partition $\mathcal{Q}$ of $V^{k+1}$, given by
        \[\left\lbrace\left\lbrace w=(w_1, ..., w_{k+1})\in V^{k+1}: w_{\neq i}\in P_i\text{ for all } i\in [k+1]\right\rbrace: P_1, ..., P_{k+1}\in\mathcal{P}\right\rbrace,\]
    such that $\sum_{Q\in\mathcal{Q}\text{ }\varphi\text{-homogeneous}} |Q|\geq (1-\delta) |V|^{k+1}$.
\end{defn}

\begin{problem}\label{qnrecoverkdistality}
    Let $\varphi(x_1, ..., x_{k+1})$ be a relation on a set $M$ that satisfies the NIP strongly $k$-distal regularity lemma. Must $(M,\varphi)$ admit an expansion that is NIP strongly $k$-distal? What if we assume that $(M, \varphi)$ is NIP?
\end{problem}
Note that, by \cite[Theorem 5.6]{mypaperblms}, when $k=1$ and $(M,\varphi)$ is not assumed to be NIP, the answer to the first question is negative.

In \cite{mypaperblms}, we also showed that relations satisfying the (NIP strongly $1$-)distal regularity lemma already enjoy a particular property of (relations definable in) distal structures, namely, improved Zarankiewicz bounds. We can ask if a similar phenomenon occurs with the NIP strongly $k$-distal regularity lemma.

\begin{problem}\label{qnpartiallyrecoverkdistality}
    Let $\varphi(x_1, ..., x_{k+1})$ be a relation on a set $M$ that satisfies the NIP strongly $k$-distal regularity lemma. Investigate the (combinatorial) properties of $\varphi$.
\end{problem}
\bibliographystyle{plainurl}
\bibliography{bib}
\end{document}